\documentclass[12pt]{amsart}
\usepackage{bbm}
\usepackage{amssymb}
\usepackage{amsmath}
\usepackage{enumitem}
\usepackage{breqn} 

\usepackage{bm}
\usepackage{bbm}

\usepackage{relsize}

\usepackage{comment}

\newcommand{\mbb}{\mathbb}
\newcommand{\mbf}{\mathbf}
\newcommand{\Z}{\mbb Z}

\newcommand{\R}{\mbb R}

\renewcommand{\epsilon}{\varepsilon}

\newcommand{\what}{\widehat}
\newcommand{\oline}{\overline}
\newcommand{\w}{\mathbf w}

\newcommand{\supp}{\mathbbm{1}}

\newcommand{\FacetsetO}[2]{{\oline{\mathbf F}^{#1}_{(#2)}(M)}}
\newcommand{\Facetset}[2]{{\mathbf F^{#1}_{(#2)}(M)}}

\newcommand{\parp}{{\Pi}}

\DeclareMathOperator{\sgn}{sgn}

\DeclareMathOperator{\tsgn}{tsgn}
\DeclareMathOperator{\wsgn}{{\mbf w}sgn}

\newcommand{\proj}{p_r}
\newcommand{\nproj}{\what p_{k}}

\DeclareMathOperator{\slice}{slice}
\newcommand{\rRestrict}{\slice_r} 

\usepackage[top=3cm,bottom=2cm,left=3cm,right=3cm,marginparwidth=2cm]{geometry}
\usepackage{amsmath,amssymb,amsthm}
\usepackage{graphicx}
\usepackage{blkarray, bigstrut} 
\usepackage[colorinlistoftodos]{todonotes}
\usepackage[colorlinks=true, allcolors=blue]{hyperref}
\usetikzlibrary{arrows.meta}
\usetikzlibrary{calc}
\usetikzlibrary{patterns}

\newtheorem{theorem}{Theorem}[section]
\newtheorem{prop}[theorem]{Proposition}

\newtheorem{open}[theorem]{Open}

\newtheorem{lemma}[theorem]{Lemma}
\newtheorem{corollary}[theorem]{Corollary}

\theoremstyle{definition}
\newtheorem{remark}[theorem]{Remark}
\newtheorem{definition}[theorem]{Definition}

\newtheorem{example}[theorem]{Example}

\usepackage{array}
\newcolumntype{C}{>{\centering\arraybackslash}p{10pt}}
\newcolumntype{X}{>{\centering\arraybackslash}p{22pt}}

\definecolor{orange}{rgb}{0.898, 0.621, 0.0}
\definecolor{skyblue}{rgb}{0.336, 0.703, 0.910}
\definecolor{bluishgreen}{rgb}{0, 0.617, 0.449}
\definecolor{yellow}{rgb}{0.937, 0.890, 0.258}
\definecolor{blue}{rgb}{0, 0.445, 0.695}
\definecolor{red}{rgb}{0.832, 0.367, 0}
\definecolor{purple}{rgb}{0.797, 0.473, 0.652}

\newcommand{\AJM}[1]{{\color{blue}\sf [#1]}} 

\title{Fragmenting any Parallelepiped into a Signed Tiling}
\author{Joseph Doolittle}
\address{Institute of Geometry, TU Graz, Graz, Austria}
\email{jdoolittle@tugraz.at }
\author{Alex McDonough}
\address{Department of Mathematics, University of California, Davis, CA 95616}
\email{amcd@ucdavis.edu}

\thanks{The first author was supported by the Austrian Science Fund FWF, Project P 33278. 
Their work was also supported by the ANR-FWF International Cooperation Project PAGCAP, funded by the FWF Project I 5788. }
\date{\today}

\begin{document}

\maketitle

\begin{abstract}

It is broadly known that any parallelepiped tiles space by translating copies of itself along its edges. 
In earlier work relating to higher-dimensional sandpile groups, the second author discovered a novel construction which fragments the parallelepiped into a collection of smaller tiles. 
These tiles fill space with the same symmetry as the larger parallelepiped.
Their volumes are equal to the components of the multi-row Laplace determinant expansion, so this construction only works when all of these signs are non-negative (or non-positive).

In this work, we extend the construction to work for all parallelepipeds, without requiring the non-negative condition.
This naturally gives tiles with negative volume, which we understand to mean canceling out tiles with positive volume.
In fact, with this cancellation, we prove that every point in space is contained in exactly one more tile with positive volume than tile with negative volume.
This is a natural definition for a signed tiling.

Our main technique is to show that the net number of signed tiles doesn't change as a point moves through space.
This is a relatively indirect proof method, and the underlying structure of these tilings remains mysterious.

\end{abstract}

\section{Introduction}

To motivate our work, we begin with an illustrative two-dimensional example of our main construction. 
Consider the matrices
\[K = \begin{bmatrix} 1 & 2 \\ -1 & 3\end{bmatrix}, \hspace{1cm} S_{\{1\}}(K) = \begin{bmatrix} 1 & 0 \\ 0 & -3 \end{bmatrix},\hspace{.5 cm} \text{ and }\hspace{.5cm}S_{\{2\}}(K) = \begin{bmatrix} 0 & 2 \\ 1 & 0 \end{bmatrix}.\]

The matrices $S_{\{1\}}(K)$ and $S_{\{2\}}(K)$ are called the \emph{fragment matrices} of $K$. 
They are obtained by negating the second row and then zeroing out a diagonal. 
Directly from the Laplace expansion for determinants, we can see that \(-\det(K) = \det(S_{\{1\}}(K)) + \det(S_{\{2\}}(K))\) 

The matrices $K$, $S_{\{1\}}(K)$, and $S_{\{2\}}(K)$ can be used to construct a periodic tiling of $\R^2$. 
Given a matrix $N$, let $\Pi(N)$ be the (half-open) \textit{fundamental parallelepiped} of $N$ (see Definition~\ref{def:halfopen} for details). 
Consider the parallelepipeds $\Pi(S_{\{1\}}(K))$ and $\Pi(S_{\{2\}}(K))$, along with their translates by all of the integer combinations of columns of $K$. 
These tiles completely fill $\R^2$ with no gaps or overlap, and produce the periodic tiling described in Figure~\ref{fig:basictile}. 

\begin{figure}
\begin{center}
\begin{tikzpicture}[scale = 0.5]
    \tikzstyle{every node} = [shape = none,fill=none,inner sep=1pt,minimum size = .1mm];
\begin{scope}
    \clip (-5,-5) rectangle (5,5);
    
\draw[dotted] (-5,5) -- (5,-5);
\draw[dotted] (-3,8) -- (7,-2);

\foreach \i in {-5,-4,-3,-2,-1,0,1,2,3,4,5}{
\foreach \j in {-2,-1,0,1,2}
{
\begin{scope}[shift = {(\i +2*\j,-\i + 3*\j)}]
\draw[fill=red!30] (0,0) -- (1,0) --(1,-3) -- (0,-3) --cycle;
\end{scope}}}

\foreach \i in {-5,-4,-3,-2,-1,0,1,2,3,4,5}{
\draw[dotted] (-5+\i*2,5 + \i*3) -- (5+\i*2,-5+\i*3);
\draw[dotted] (-20+\i,-30 - \i) -- (20+\i,30 - \i);}

\draw[line width = 1 mm] (0,0) -- (1,0) --(1,-3) -- (0,-3) --cycle;
\node[shape = circle, minimum size = 2 mm, fill = black] at (0,0){};
\end{scope}
\begin{scope}[shift = {(12,0)}]
    \clip (-5,-5) rectangle (5,5);
    
\foreach \i in {-5,-4,-3,-2,-1,0,1,2,3,4,5}
\foreach \j in {-2,-1,0,1,2}
{
\begin{scope}[shift = {(\i +2*\j,-\i + 3*\j)}]
\draw[fill=blue!20] (0,0) -- (2,0) -- (2,1) -- (0,1) -- cycle;
\end{scope}}
\draw[line width = 1 mm] (0,0) -- (2,0) --(2,1) -- (0,1) --cycle;
\node[shape = circle, minimum size = 2 mm, fill = black] at (0,0){};

\foreach \i in {-5,-4,-3,-2,-1,0,1,2,3,4,5}{
\draw[dotted] (-5+\i*2,5 + \i*3) -- (5+\i*2,-5+\i*3);
\draw[dotted] (-20+\i,-30 - \i) -- (20+\i,30 - \i);}
\end{scope}
\end{tikzpicture}
\caption{This is the tiling obtained by translating the fundamental parallelepipeds of $S_{\{1\}}(K)$ and $S_{\{2\}}(K)$ by integer linear combinations of the columns of $K$. Translates of $S_{\{1\}}(K)$ are given on the left in orange and translates of $S_{\{2\}}(K)$ are given on the right in blue. When these partial tilings are combined, we get a full periodic tiling of $\R^2$. }
\label{fig:basictile}
\end{center}
\end{figure}
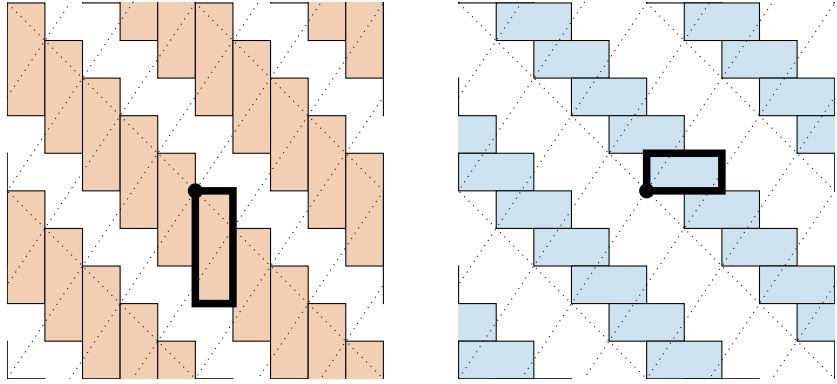


This tiling is a two dimensional example of a construction which was introduced by the second author to define \textit{matrix-tree multijections}~\cite{multijections,alexthesis}. 
This construction can be applied to any invertible $(r+k)\times(r+k)$ matrix $M$, and produces a collection of $\binom{r+k}{r}$ \emph{fragment matrices} of $M$. 
When the determinants of the fragment matrices are all non-negative (or all non-positive), translating them by integer linear combinations of the columns of $M$ produces a periodic tiling of $\R^{r+k}$. 

In this paper, we prove that the elegant tiling structure of the fragment matrices is still present even without the restriction on $M$ that all the fragment matrices have non-negative determinant. 
In particular, while the translates do not always form a traditional tiling with no overlap or gaps, they always produce a \emph{signed tiling}.

To illustrate this signed version of the tiling, we give another $2$-dimensional example. 
This time, the determinants of the fragment matrices have opposite signs. 

Let
\[L = \begin{bmatrix} 1 & 2 \\ 1 & 5\end{bmatrix}, \hspace{1cm} S_{\{1\}}(L) = \begin{bmatrix} 1 & 0 \\ 0 & -5 \end{bmatrix},\hspace{.5 cm} \text{ and }\hspace{.5cm}S_{\{2\}}(L) = \begin{bmatrix} 0 & 2 \\ -1 & 0 \end{bmatrix}.\]

As in the previous example, the fragment matrices $S_{\{1\}}(L)$ and $S_{\{2\}}(L)$ are formed by negating the second row and zeroing a diagonal. 
Next, we consider translates of the fragment matrices by integer linear combinations of the columns of $L$. 
In this case, the tiles no longer perfectly fill space, and instead overlap, see Figure~\ref{fig:signedtile}.

In our previous example, the determinants of $S_{\{1\}}(K)$ and $S_{\{2\}}(K)$ were both positive.
In this example, $S_{\{1\}}(L)$ is positive, but $S_{\{2\}}(L)$ is negative.
Moreover, the positively signed tiles overlap.
Nevertheless, an elegant tiling structure can still be found. 

\begin{figure}
\begin{center}
\begin{tikzpicture}[scale = 0.5]
    \tikzstyle{every node} = [shape = none,fill=none,inner sep=1pt,minimum size = .1mm];
    
\begin{scope}
    \clip (-5,-5) rectangle (5,5);
\foreach \i in {-20,...,20}
\foreach \j in {-4,...,4}
{
\begin{scope}[shift = {(\i +2*\j,\i + 5*\j)}]
\draw[fill=red!40] (0,0) -- (0,1) --(1,1) -- (1,0) --cycle;
\draw[fill=red!40] (1,3) -- (1,4) --(0,4) -- (0,3) --cycle;
\draw[fill=red!100] (0,1) -- (0,3) -- (1,3) -- (1,1) -- cycle;
\end{scope}}
\node[shape = circle, minimum size = 2 mm, fill = black] at (0,0){};
\draw[line width = 1 mm] (0,0) -- (1,0) --(1,-5) -- (0,-5) --cycle;

\foreach \i in {-10,-9,-8,-7,-6,-5,-4,-3,-2,-1,0,1,2,3,4,5,6,7,8,9,10}{
\draw[dotted] (-20+\i*2,-20 + \i*5) -- (20+\i*2,20 + \i*5);
\draw[dotted] (-20+\i,-50 + \i) -- (20+\i,50 + \i);}
\end{scope}

\begin{scope}[shift = {(12,0)}]
    \clip (-5,-5) rectangle (5,5);
\foreach \i in {-20,...,20}
\foreach \j in {-4,...,4}
{
\begin{scope}[shift = {(\i +2*\j,\i + 5*\j)}]
\draw[fill=blue!40] (0,0) -- (2,0) -- (2,-1) -- (0,-1) -- cycle;
\end{scope}}
\node[shape = circle, minimum size = 2 mm, fill = black] at (0,0){};
\draw[line width = 1 mm] (0,0) -- (2,0) --(2,-1) -- (0,-1) --cycle;

\foreach \i in {-10,-9,-8,-7,-6,-5,-4,-3,-2,-1,0,1,2,3,4,5,6,7,8,9,10}{
\draw[dotted] (-20+\i*2,-20 + \i*5) -- (20+\i*2,20 + \i*5);
\draw[dotted] (-20+\i,-50 + \i) -- (20+\i,50 + \i);}
\end{scope}

\end{tikzpicture}
\caption{On the left is the tiling obtained by translating the fundamental parallelepipeds of $S_{\{1\}}(L)$ by integer combinations of columns of $L$. 
The darker regions indicate where two parallelepipeds overlap, while the lighter region is the portion covered by a single parallelepiped. 
On the right is the tiling obtained by translating the fundamental parallelepipeds of $S_{\{2\}}(L)$ by integer combinations of columns of $L$. 
This time, there are no overlaps, but the white region is formed by gaps between parallelepipeds. 
By Theorem~\ref{thm:mainthm}, the shaded region on the right precisely corresponds to the darker region on the left.}
\label{fig:signedtile}
\end{center}
\end{figure}
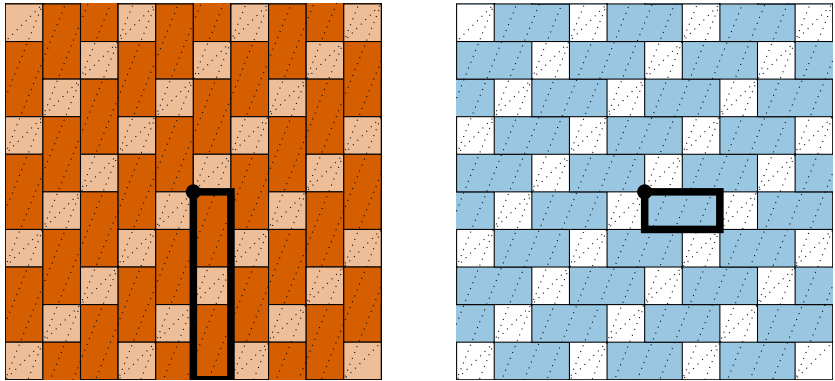

Consider the two partial tilings given in Figure~\ref{fig:signedtile}. Every point in the plane is covered by either one translate of $\Pi(S_{\{1\}}(L))$ or two translates of $\Pi(S_{\{1\}}(L))$ and one translate of $\Pi(S_{\{2\}}(L))$. 
This means that if we define translates of $\Pi(S_{\{1\}}(L))$ to be \textit{positive} tiles and translates of $\Pi(S_{\{2\}}(L))$ to be \textit{negative} tiles, then for any point $\mbf p \in \R^2$, the signed total of all tiles containing $\mbf p$ is always $1$. 

This surprising alignment of positive and negative tiles works in general. 
Reiterating the previous setting, we let \(M\) be an invertible \((r+k)\times (r+k)\) matrix.
We break this matrix into two parts, the first \(r\) rows and the last \(k\) rows.
The two tiles from the two dimensional case become \(\binom{r+k}{r}\) many tiles, indexed by which \(r\) columns are preserved in the top $r$ rows (see Definition~\ref{def:fragment} for details). 

We generalize the cancellation observed in the example with \(L\) by introducing a function \(f\). This function counts the number of positively signed tiles at a point, minus the number of negatively signed tiles at that point.

\vspace{2 ex}
\noindent\textbf{Theorem~\ref{thm:mainthm}.} \textit{
 The function \(f(\mbf p) : \R^{r+k} \to \Z\), defined by
    \begin{equation*}f(\mbf p) := \left(\sum_{T \in \mathbf T^+(M)}  \supp_{T}(\mbf p)\right) - \left(\sum_{T \in \mathbf T^-(M)}  \supp_{T}(\mbf p)\right),\end{equation*}
    is constant with value $(-1)^{k}\sgn(\det(M))$. 
}

In Section~\ref{sec:def}, we provide a glossary of all the notation we use. We then follow it with definitions and notation, which apply in more general settings than this work.
In Section~\ref{sec:construction}, we describe the general construction, introduce some specialized definitions, and state the main theorem. To prove the theorem, we develop two lemmas in the following sections
In Section~\ref{sec:average}, we compute the average value of \(f\).
In Section~\ref{sec:crossing}, we show that when crossing the boundary of a tile, the value of the function \(f\) is constant. This section is by far the longest and most technical.
In Section~\ref{sec:together}, we combine our previous work and deliver the proof of the main theorem.
In Section~\ref{sec:proj}, we explore projections of these tilings, and give a glimpse of the high dimensional structure in a two dimensional document. 
Finally, in Section~\ref{sec:questions}, we consider future extensions and pose questions we think will be interesting to explore.


\section{Definitions and Notation}\label{sec:def}

\subsection{Glossary}

Many of the proofs in the paper require a large amount of notation, so we included this glossary to help the reader keep track of it all. 
This is hopefully as helpful to the readers as it has been for the authors.

\begin{itemize}
\item $r$ and $k$ are positive integers which we fix throughout the paper. \(r+k\) is the dimension of real space that we consider, as well as the size of \(M\).

\item $M$ is an invertible $(r+k)\times (r+k)$ real matrix which we fix throughout the paper. 

\item \(\w\) is a vector in $\R^{r+k}$ that we fix throughout the paper. We require $\w$ to be ``sufficiently generic'', meaning that it has no linear dependence with any of the many vectors we consider. See Section~\ref{sec:fixedvalues}.

\item $\w'$ and $\w''$ are the vectors obtained by restricting $\w$ to its first $r$ or last $k$ coordinates respectively. By construction, these vectors are also ``sufficiently generic'' in their respective spaces. 

\item $\sigma$, $\tau$, and $\gamma$ are all variables used for subsets of $[r+k]$. They will typically be used to represent subsets of size $r$, $r-1$, and $r+1$ respectively. 

\item $\what\sigma$ is the set $[r+k] \setminus \sigma$ for any $\sigma \subseteq [r+k]$. 

\item \(\mbf q\), \(\mbf p\), \(\mbf x\), and $\mbf v$ appear with and without indices, and are variable real vectors used in proofs. 

\item \(\mbf 0\) is the all zero vector.

\item \(\mbf z\) is a variable used for an integer vector. 

\item \(\parp(N)\) is the \emph{half-open parallelepiped} of the square matrix \(N\) with orientation \(\mbf w\). See Definition~\ref{def:halfopen}.

\item \(Z(N)\) is the \emph{zonotope} of the matrix \(N\) with orientation \(\w\), a generalization of the half-open parallelepiped. See Definition~\ref{def:halfopen}.

\item \(\mbf b_i\) and \(\mbf c_i\) (for $i \in [r+k]$) appear with and without overlines, and are fixed real vectors derived from the columns \(M\). See Section~\ref{sec:matrixdecomp}.

\item \(S_{\sigma}(M)\) is a \emph{fragment matrix} of $M$. Its columns are \(\mbf b_i\) if \(i \in \sigma\) and \(\oline{\mbf b}_i\) otherwise. See Definition~\ref{def:fragment}. 

\item \(C_{\sigma}(M) \) and \(\oline{C}_{\what\sigma}(M) \) are particular submatrices of \(S_{\sigma}(M)\). See Definition~\ref{def:ccbar}.

\item \({\bm \lambda}^\sigma\) is the representation of \(\w\) in the basis given by \(S_\sigma(M)\). See~\eqref{eq:lambda}.

\item \(\mbf h\) is a vector in the kernel of \(\oline{C}_{\what{\tau}}(M)\) depending only on \(\mbf w\) and \(\tau\). See Proposition~\ref{prop:h}.

\item $\mathcal T(\mbf z, \sigma)$ is the tile given by translating $\parp(S_\sigma(M))$ by $M\mbf z$.  See Definition~\ref{def:tile}.

\item \( \mbf T (M)\) is the set of all tiles in our construction (which are formed by translations parallelepipeds formed by fragment matrices). See Definition~\ref{def:tilesets}.

\item \( \mbf T^+(M)\) is the set of ``positive tiles'' in our construction, i.e., tiles whose corresponding fragment matrix has positive determinant. See Definition~\ref{def:tilesets}.

\item \( \mbf T^-(M)\) is the set of ``negative tiles'' in our construction, i.e., tiles whose corresponding fragment matrix has positive determinant. See Definition~\ref{def:tilesets}.


\item $\mathcal F(\mbf z,\sigma,j,s)$ is a facet of the tile $\mathcal T(\mbf z, \sigma)$, missing the \(j\)th vector, and on the \(s\) side of the tile. See Definition~\ref{def:facets}.

\item $\widetilde{\mathcal F}(\mbf z,\sigma,j,s)$ is an alternative parameterization of the facets, so they are easier to collect into hyperplanes. See Definition~\ref{def:tildefacet}.

\item \( \Facetset{~}{\mbf z, \tau}\) is a collection of facets of tiles which lie in a common hyperplane determined by \(\mbf z\) and \(\tau\). See Definition~\ref{def:facetcollections}.

\item \( \FacetsetO{~}{\mbf z, \gamma}\) is a collection of facets of tiles which lie in a common hyperplane determined by \(\mbf z\) and \(\gamma\). See Definition~\ref{def:facetcollections}.

\item \( \Facetset{\downarrow}{\mbf z,\tau} \) is a collection of facets in \( \Facetset{~}{\mbf z, \tau}\) which are facets of tiles which decrease the value of \(f\) when crossing the facet in the direction of \(\w\). See Definition~\ref{def:FupFdown}.

\item \( \Facetset{\uparrow}{\mbf z,\tau} \) is a collection of facets in \( \Facetset{~}{\mbf z, \tau}\) which are facets of tiles which increase the value of \(f\) when crossing the facet in the direction of \(\w\).

\item $\proj$ is the map that projects \(\R^{r+k}\) to the first $r$ coordinates.

\item $\nproj$ is the map that projects \(\R^{r+k}\) to the last $k$ coordinates. 

\item \(\wsgn\) is a function that detects if a facet is open or closed in the \(\w\) direction. See Definition~\ref{def:wsgn}.

\item \(\tsgn\) is a function that detects the sign of the determinant of the tile containing a facet. See Definition~\ref{def:tsgn}.

\item $\bigsqcup$ and $\biguplus$ both indicate kinds of \emph{disjoint union}. See Section~\ref{sec:distunion} for more details. 
\end{itemize}

\subsection{Matrices, Vectors, and Fixed Values}\label{sec:fixedvalues}

We write matrices using capital letters and will denote the $i^{th}$ column of $N$ by $N_i$. 
Otherwise, we will use lowercase bold font to indicate vectors, which will always be column vectors. As a general rule, if we write a bold lowercase letter with a subscript, e.g., $\mbf x_i$, then this refers to a vector among a collection of vectors. However, a non-bold lowercase letter with a subscript, e.g., $x_i$, might be used to refer to an entry of a particular vector. We write $\mbf 0$ for the all zeros vector and $\mbf e_i$ for the $i^{th}$ standard basis vector. 
Whenever we give a sum of intervals, this is always interpreted as Minkowski sum. 

Throughout this paper, we fix two integers $r,k$ to be global dimension parameters. 
In particular, \(r+k\) is what would normally be named \(d\), the ambient dimension.
We also fix an invertible $(r+k)\times (r+k)$ matrix $M$. 
Furthermore, we fix a \emph{sufficiently generic} vector \(\w\in \R^{r+k}\).
By \emph{sufficiently generic}, we assume that for any of the finitely many matrices we consider, $\w$ is not in the span of any collection of $r+k-1$ of their column vectors. 
In particular, for any matrix \(N\) under our consideration, all of the entries in the vector \(N^{-1}\w\) are non-zero.

Several of the objects that we study depend on the vector $\w$, but this dependency will disappear, so our results are all independent of this choice. 
For simplicity of notation, we typically omit $\w$ when defining objects. We also write $\w'$ and $\w''$ for the vectors made up of the first $r$ and last $k$ entries of $\w$ respectively. 

\subsection{Subsets and Signs}
We write $[n]$ for the set $\{1,2,\dots,n\}$, and $\binom{[n]}{i}$ for the subsets of $[n]$ of size $i$. 
Given some $\sigma \subseteq [n]$, we write $\what\sigma$ for the set $[n] \setminus \sigma$. 
We typically use the variable $\sigma$ to denote an element of $\binom{[r+k]}{r}$, the variable $\tau$ to denote an element of $\binom{[r+k]}{r-1}$, and the variable \(\gamma\) to denote an element of $\binom{[r+k]}{r+1}$. 
We will always consider subsets of the natural numbers to be ordered from smallest to largest, and write $\sigma(i)$ for the $i^{th}$ smallest element of $\sigma$ (beginning at 1). 
For example, let $\sigma = \{1,4,5\}$ be thought of as a subset of $[5]$, i.e., as an element of $\binom{[5]}{3}$. 
Then, $\sigma(1) = 1$, $\sigma(2) = 4$, and $\sigma(3) = 5$. 
Furthermore, we also have $\what \sigma = \{2,3\}$, with $\what\sigma(1) = 2$ and $\what\sigma(2) = 3$.

In this paper, we come across several different sign functions, which all output either $1$ or $-1$. 
For $x \in \R \setminus \{0\}$ we define $\sgn(x)$ to output $1$ if $x>0$ and $-1$ if $x<0$, the standard sign function. 

We are also interested in the sign of permutations. 
For a permutation $\rho$ of the set $[n]$, we define $\sgn(\rho)$ to be the parity of the number of transpositions. 
We will typically express a permutation as a list of disjoint subsets whose union is $[r+k]$. Continuing the example from the previous paragraph, suppose we wanted to compute $\sgn(\sigma\setminus \{4\},\{4\},\what\sigma)$. We would write this computation as \[\sgn(\sigma\setminus \{4\}, \{4\}, \what\sigma) =\sgn(\{1,5\}, \{4\},\{2,3\}) = \sgn(1,5,4,2,3) = -1,\]
since five transpositions are needed to reach the identity. 

Finally, there are two special sign functions, $\wsgn$ and $\tsgn$, whose definitions are given in Section~\ref{sec:updown}. 

\subsection{Two Kinds of Disjoint Union}\label{sec:distunion}

Let us now take a moment to clarify some confusing notation. The term \emph{disjoint union} can be used in two different ways in mathematics, and both perspectives are used in this paper. To minimize confusion, we will denote the two versions of disjoint union with two different symbols. For sets $A$ and $B$, we write $A\bigsqcup B$ for the set $A\cup B$ with the added restriction that $A\cap B = \emptyset$. We use the notation $A \biguplus B$ to indicate the other kind of disjoint union, where $A$ and $B$ are considered as separate objects.

For example, suppose that $A$ and $B$ are polytopes in $\R^{r+k}$. Then, $A \bigsqcup B$ is the set of points in $\R^{r+k}$ that are contained in either $A$ or $B$.  while we will usually think of $A \biguplus B$ as a set made up of two polytopes. 

\subsection{Parallelepipeds and Translations }\label{sec:parpipeds}

\begin{definition}\label{def:halfopen} 
Let $N$ be an $(r+k) \times m$ matrix for some $m \in \Z^+$. We define $Z(N)$ to be the set of $\mbf p \in \R^{r+k}$ such that for all sufficiently small $\epsilon >0$, the point $\mbf p + \epsilon\w$ is in 
\begin{equation}\label{eq:closure}\sum_{i \in [m]}\left\{x_iN_i: 0 \le x_i \le 1\right\}.\end{equation}
The set $Z(N)$ is called the \emph{(half-open column) zonotope of $N$}. 

We also sometimes work with matrices that have $r$ or $k$ rows instead of $r+k$ rows. In these cases, $Z(N)$ is defined analogously, but with $\w$ replaced with $\w'$ or $\w''$ respectively. 

When $N$ is a square matrix, we use the notation $\parp (N)$ in place of $Z(N)$. 
In this case, the zonotope $\parp(N)$ is the \emph{(half-open) parallelepiped} generated by $N$.

\end{definition}

In words, $\parp(N)$ is the parallelepiped formed by the columns of $N$, with half of the boundary removed. 
The boundary points of the closure of $\parp(N)$ are included in $\parp(N)$ if and only if shifting them by an arbitrary small positive amount in the direction of $\w$ maps into the interior of $\parp(N)$. 
Note that when $N$ is invertible, the parallelepiped $\parp(N)$ is $n$-dimensional and its volume ($n$-dimensional Lebesgue measure) is equal to the absolute value of the determinant of $N$. 
These definitions are given with specified dimension, but analogous definitions are appropriate in any dimension.

\begin{lemma}\label{lem:halfopen}

Suppose $N$ is an $(r+k)\times(r+k)$ matrix. If $N$ is invertible, then
\[\parp(N) = \sum_{i\in [r+k]} \left\{ x_i N_i : \begin{array}{l}0 < x_i \leq 1 \text{ if $(N^{-1}\w)_i < 0$}\\0 \leq x_i<  1  \text{ if $(N^{-1}\w)_i > 0$}\end{array} \right\},
\]
If $N$ is not invertible, then $\parp(N) = \emptyset$.
\end{lemma}

\begin{proof}

    First, we consider the case where $N$ is not invertible. This implies that \eqref{eq:closure} is less than $(r+k)$-dimensional. This means that shifting any point in this region in a sufficiently generic direction must leave the region. In particular, $\parp(N) = \emptyset$.


If instead, \(N\) is invertible, we define $\mbf y = (y_1,\dots,y_{r+k})^\top= N^{-1}\mbf p$ for some \(\mbf p \in \R^{r+k}\). 
Then, for any real number $\epsilon$, it follows that
\[\mbf p + \epsilon\w = N \mbf y + \epsilon\w = N \mbf y + \epsilon (N (N^{-1} \w)) = N  (\mbf y + \epsilon N^{-1} \w) = \sum_{i \in [r+k]} (y_i + \epsilon(N^{-1}\w)_i)N_i.\]

By definition, $\mbf p + \epsilon\w$ is in the region defined in \eqref{eq:closure} if and only if for all $i$, we have $0 \le y_i + \epsilon(N^{-1}\w)_i \le 1$. This is true for all sufficiently small positive $\epsilon$ if $0 < y_i \le 1$ when  $(N^{-1}\w)_i<0$, and $0 \le y_i < 1$ when  $(N^{-1}\w)_i>0$
    \end{proof}


We present a simple observation about translating parallelepipeds, which will be the foundation of our construction.

\begin{lemma}\label{lem:basictiling}
    For any choice of $M$, we have 
    \begin{equation}\label{eq:basictiling}
    \mbb{R}^{r+k} = \bigsqcup_{\mbf {z} \in \mbb {Z}^{r+k}} \left(\parp(M) + M\mbf z\right).
    \end{equation}
\end{lemma}

This lemma follows from the fact that the unit cube tiles space, and the displacement between cubes in this tiling is all \(\mbb Z\)-valued vectors.
The lemma describes this same tiling, after applying \(M\) as a linear transformation.

\begin{definition}
    For $S \subseteq \R^{r+k}$, let $\supp_S$ be the \emph{support function} of $S$. In particular, for $\mbf p \in \R^{r+k}$, we have 
    \[\supp_S(\mbf p) = \begin{cases}1 & \text{if $\mbf p \in S$,}\\
    0 & \text{otherwise.}\end{cases}\]
\end{definition}

The following elementary lemma gives an alternate way to express the idea that a collection of subsets cover $\R^{r+k}$ with no gaps or overlaps. 
\begin{lemma}\label{lem:support} 
    Let $\mathbf S$ be a collection of subsets of $\R^{r+k}$. 
    \[ \mbb R^{r+k} = \bigsqcup_{S \in \mathbf S} S\hspace{1 cm} \iff \hspace{1cm} \forall \mbf p \in \R^{r+k},~ \sum_{S \in \mathbf S} \supp_{S}(\mbf p) = 1.\]
\end{lemma}

Combining Lemmas~\ref{lem:basictiling} and~\ref{lem:support}, we can immediately conclude that for any $\mbf p \in \R^{r+k}$, we have
\[\sum_{\mbf {z} \in \mbb {Z}^{r+k}} \supp_{(\parp(M) + M\mbf z)}(\mbf p) = 1.\]

\section{Signed Tiling Construction and Main Result}\label{sec:construction}
\subsection{Matrix Decompositions}\label{sec:matrixdecomp}

Recall that we fix positive integers $r$ and $k$ as well as an $(r+k)\times (r+k)$ matrix $M$. 
For \(i \in [r+k]\), we write \(\mbf c_i\) for the column vector of length \(r\) consisting of the first \(r\) entries of \(M_i\). 
We write \(-\oline {\mbf c}_i\) for the column vector of length \(k\) consisting of the last \(k\) entries of \(M_i\). 
The negative sign in the previous definition is slightly unexpected, but is required for the construction to work.

Additionally, for each $i \in [n]$, we write $\mbf b_i$ for the column vector of length $r+k$ whose first $r$ entries are the same as those in $\mbf c_i$, and whose last $k$ entries are $0$. 
Similarly, we write $-\oline{\mbf b}_i$ for the column vector whose first $r$ entries are $0$ and whose last $k$ entries are the same as those in $-\oline{\mbf c}_i$.
We summarize these definitions by the two following decompositions of \(M\). 
Note that the vertical lines are intended to help show the structure of the matrix, but do not have any mathematical meaning. 

\[M = \begin{bmatrix}\vline & \vline &  & \vline \\\mbf c_1 & \mbf c_2 & \hdots & \mbf c_{r+k}\\
\vspace{3pt}\vline & \vline &  & \vline \\
\vline & \vline &  & \vline \\
-\oline {\mbf c}_1 & -\oline {\mbf c}_2 & \hdots & -\oline {\mbf c}_{r+k}\\
\vline & \vline &  & \vline
\end{bmatrix} =  \begin{bmatrix}
\vrule height 33pt depth 0pt width 0pt \vline & \vline &  & \vline \\
\mbf b_1 &\mbf b_2 & \hdots & \mbf b_{r+k}\\
\vrule height 31pt depth 0pt width 0pt \vline & \vline &  & \vline \\
\end{bmatrix} + \begin{bmatrix}
\vrule height 33pt depth 0pt width 0pt \vline & \vline &  & \vline \\
-\oline{\mbf b}_1 & -\oline{\mbf b}_2 & \hdots & -\oline{\mbf b}_{r+k}\\
\vrule height 31pt depth 0pt width 0pt \vline & \vline &  & \vline \\
\end{bmatrix}.  \]

\begin{definition}\label{def:fragment}
Let $\sigma \in \binom{[r+k]}{r}$. The \emph{\(\sigma\)-fragment matrix} of \(M\), written $S_\sigma(M)$, is the matrix defined by
\[ S_\sigma(M)_i := \begin{cases} \mbf b_i & \text{ if $i \in \sigma$,}\\
\oline{\mbf b}_i & \text{ if $i \not \in \sigma$.} \end{cases}\]
We will sometimes refer to the parallelepiped $\parp(S_\sigma(M))$ as a \emph{fragment} of $M$.
\end{definition}

In other words, $S_\sigma(M)$ is the matrix obtained from $M$ by the following $3$ step process:
\begin{enumerate}
    \item For each $i \in \what\sigma$, replace the \emph{first} $r$ entries of column $i$ with $0$. 
    \item For each $i \in \sigma$, replace the \emph{last} $k$ entries of column $i$ with $0$. 
    \item Negate all of the entries in the last $k$ rows. 
\end{enumerate}

\begin{definition}\label{def:ccbar}
We will also work with the $r \times |\sigma|$ matrix $C_\sigma(M)$ and the $k \times |\what\sigma|$ matrix $\oline{C}_{\what \sigma}(M)$ which are defined by
\[ C_\sigma(M) := \begin{bmatrix}\vert & \vert &  & \vert \\\mbf c_{\sigma(1)} & \mbf c_{\sigma(2)} & \hdots & \mbf c_{\sigma(|\sigma|)}\\\vspace{.15 cm}
\vert & \vert &  & \vert
\end{bmatrix} \hspace{2cm}\oline{C}_{\what\sigma}(M) := \begin{bmatrix}\vert & \vert &  & \vert \\ \oline{\mbf c}_{\what\sigma(1)} &  \oline{\mbf c}_{\what \sigma(2)} & \hdots &  \oline{\mbf c}_{\what \sigma(|\what \sigma|)}\\\vspace{.15 cm}
\vert & \vert &  & \vert
\end{bmatrix}\]
Recall that we sometimes treat $\sigma$ as a list ordered from smallest to largest, and write $\sigma(i)$ to denote the $i^{th}$ entry in this list.
Note that the fragment matrices will always have $|\sigma| = r$, but the more general definition of $C_\sigma(M)$ and $\oline{C}_{\what\sigma}(M)$ will prove useful in Section~\ref{sec:crossing}.

\end{definition}

\begin{example}\label{ex:first}

Throughout this paper, we will consider the running example with $r = k = 2$ and \[M = \begin{bmatrix} 
3 & 2 & -4 & 1\\
1 & 0 & 2 & 2\\
2 & 0 & -1 & 1\\
0 & 1 & -2 & 3\\
\end{bmatrix}.\]
The set $\binom{[4]}{2}$ contains $6$ elements, so there are $6$ different fragments. For example, when $\sigma = \{1,4\}$, we have 
\[S_{\sigma}(M) =\begin{bmatrix} 
3 & 0 & 0 & 1\\
1 & 0 & 0 & 2\\
0 & 0 & 1 & 0\\
0 & -1 & 2 & 0\\
\end{bmatrix}, \hspace{.3 cm}C_\sigma(M)  = \begin{bmatrix} 
3 & 1\\
1 & 2\\
\end{bmatrix}, \hspace{.3 cm} \text{ and } \hspace{.3 cm} \oline{C}_{\what\sigma}(M) =\begin{bmatrix} 
0 & 1 \\
-1 &  2\\
\end{bmatrix}.\]

\end{example}
\subsection{Signed Tiling Construction}
To form a signed tiling, we parameterize tiles formed by translating the fundamental parallelepiped of fragment matrices by integer combinations of the columns of $M$. 

\begin{definition}\label{def:tile}
Recall Definitions~\ref{def:halfopen} and \ref{def:fragment}.
For any $\mbf z \in \Z^{r+k}$ and $\sigma \in \binom{[r+k]}{r}$, the \emph{tile} parameterized by the pair \((\mbf z, \sigma)\) is defined as
    \[\mathcal T(\mbf z, \sigma) := \parp(S_{\sigma}(M))+M\mbf z.\]
\end{definition}

Note that since $\parp(S_{\sigma}(M))$ depends on $\w$, the tile $\mathcal T(\mbf z,\sigma)$ will depend on $\w$ as well. Nevertheless, the precise choice of $\w$ is not important for our results as long as it remains fixed (and sufficiently generic, see Section~\ref{sec:fixedvalues}). Also, note that we usually think of a tile $\mathcal T(\mbf z, \sigma)$ as a polytope made up of a collection of points, not the points themselves. With this perspective in mind, we introduce the following definition. 

\begin{definition}\label{def:tilesets} 
Consider the sets of tiles
\begin{align*}{\mathbf T}^+(M) :&= \biguplus_{\mbf z \in \Z^{r+k}} \left( \biguplus_{\sigma \in \binom{[r+k]}{r},~\det(S_\sigma(M))>0} \mathcal T(\mbf z, \sigma) \right),\\
\text{and }{\mathbf T}^-(M) :&=  \biguplus_{\mbf z \in \Z^{r+k}}  \left( \biguplus_{ \sigma \in \binom{[r+k]}{r},~\det(S_\sigma(M))<0} \mathcal T(\mbf z, \sigma) \right). \end{align*}
The set ${\mathbf T}^+(M)$ is the set of \emph{positive tiles}, and  ${\mathbf T}^-(M)$ is the set of \emph{negative tiles}. 
We also write ${\mathbf T}(M) := {\mathbf T}^+(M) \biguplus {\mathbf T}^-(M)$. 
Note that we don't include the tiles where \(\det(S_\sigma(M))=0\), but in this case, \(S_\sigma(M)\) is not invertible, and \(\parp(S_\sigma(M))\) is empty.
\end{definition}

\begin{remark}In~\cite{multijections} and~\cite{alexthesis}, the word \emph{tile} is used differently, instead referring to the union of a particular collection of parallelepipeds.

\end{remark}
Definition~\ref{def:tilesets} allows us to cleanly state our main result. 

\begin{theorem}\label{thm:mainthm} The function \(f(\mbf p) : \R^{r+k} \to \Z\), defined by
    \begin{equation}\label{eq:mainthm}f(\mbf p) := \left(\sum_{T \in {\mathbf T}^+(M)}  \supp_{T}(\mbf p)\right) - \left(\sum_{T \in {\mathbf T}^-(M)}  \supp_{T}(\mbf p)\right),\end{equation}
    is constant with value $(-1)^{k}\sgn(\det(M))$. 
\end{theorem} 

\begin{remark}
    The theorem is also true in the more general setting when \(M\) is not invertible, under the convention that  \(\sgn(0) = 0\).
    For the convenience of talking about tilings, we don't discuss this generalization.
    We invite an interested reader to follow along and see exactly where we require invertibility, and that it isn't required for the proof of Theorem~\ref{thm:mainthm}.
\end{remark}

We provide the proof of Theorem~\ref{thm:mainthm} in Section~\ref{sec:together}. 
To enable the proof, we need two significant lemmas. 
The first lemma is proven in Section~\ref{sec:average}, while the second is proven in Section~\ref{sec:crossing}.
These lemmas are:
\begin{enumerate}
    \item The average of \(f(\mbf p)\) in \eqref{eq:mainthm} over the domain \(\parp(M)\) is $(-1)^{k}\sgn(\det(M))$. See Corollary \ref{cor:avgvalue}.
    \item Moving \(\mbf p\) between tiles doesn't change the value of \(f(\mbf p)\). See Theorem~\ref{thm:1pointconstant}.
\end{enumerate}

\begin{example}\label{ex:posnegtiles}
For the matrix $M$ from Example~\ref{ex:first}, the set ${\mathbf T}(M)$ consists of $6$ families of $4$-dimensional parallelepipeds, where each family contains infinitely many translations of a single fragment. 

By taking the determinant of each fragment, we find that \begin{align*}{\mathbf T}^+(M) &= \biguplus_{\mbf z \in \Z^{r+k}} \left( \biguplus_{\sigma \in \{(1,2),(1,3),(1,4),(2,3),(2,4)\}} \mathcal T(\mbf z, \sigma) \right),\text{ and}\\
{\mathbf T}^-(M) &= \biguplus_{\mbf z \in \Z^{r+k}} \mathcal T(\mbf z, \{3,4\}). \end{align*}

Confirming that Theorem~\ref{thm:mainthm} holds for this example is not a completely straightforward task, even with the help of a computer. Nevertheless, regardless of the choice of $\w$, one can show that each $\mbf p \in\R^4$ is contained in 
\begin{itemize}
    \item one tile in ${\mathbf T}^+(M)$ and no tiles in ${\mathbf T}^-(M)$,
    \item two tiles in ${\mathbf T}^+(M)$ and one tile in ${\mathbf T}^-(M)$, or
    \item three tiles in ${\mathbf T}^+(M)$ and two tiles in ${\mathbf T}^-(M)$.
\end{itemize}
In each case, the value of $f(\mbf p)$ is $1$, which is also the sign of $\det(M)$. A method for visualizing this tiling is described in Example~\ref{ex:projtiles2}.

For an even more concrete example, consider the point $(-2,1,-\frac12,-\frac12)^\top$. One can calculate that this point is on the interior of two positive tiles and one negative tile. In particular, 
\begin{align*}
    \!\left[\begin{array}{X} -2\\1\\ -1/2 \\ -1/2 \end{array}\!\right] &= 
    \!\left[\begin{array}{CCCC} 0&2&-4&0\\0 & 0 & 2 & 0\\-2 & 0 & 0 & -1\\0 & 0 & 0 & -3  \end{array}\!\right]
    \!\left[\begin{array}{X} 5/6\\ 1/2\\ 1/2\\ 5/6 \end{array}\!\right] + 
    \!\left[\begin{array}{CCCC} 3 & 2&-4&1\\1 & 0 & 2 & 2\\2 & 0 & -1 & 1\\0 & 1 & -2 & 3 \end{array}\!\right]
    \!\left[\begin{array}{C} 0\\ -3\\ -1\\ 1 \end{array}\!\right]
    \in \mathcal T\left(\left[\begin{array}{C} 0\\ -3\\ -1\\ 1 \end{array}\right],\{2,3\}\right)\\
    \!\left[\begin{array}{X} -2\\1\\-1/2\\-1/2 \end{array}\!\right] &= 
    \!\left[\begin{array}{CCCC} 0&2&0&1\\0 & 0 & 0 & 2\\-2 & 0 & 1 & 0\\0 & 0 & 2 & 0 \end{array}\!\right]
    \!\left[\begin{array}{X} 5/8\\ 3/4\\ 3/4\\ 1/2 \end{array}\!\right] + 
    \!\left[\begin{array}{CCCC} 3 & 2&-4&1\\1 & 0 & 2 & 2\\2 & 0 & -1 & 1\\0 & 1 & -2 & 3 \end{array}\!\right]
    \!\left[\begin{array}{C} 0\\ -2\\ 0\\ 0 \end{array}\!\right]
    \in \mathcal T\left(\left[\begin{array}{C} 0\\ -2\\ 0\\ 0 \end{array}\right],\{2,4\}\right)\\
    \!\left[\begin{array}{X} -2\\1\\-1/2\\-1/2 \end{array}\!\right] &= 
    \!\left[\begin{array}{CCCC} 0&0&-4&1\\0 & 0 & 2 & 2\\-2 & 0 & 0 & 0\\0 & -1 & 0 & 0 \end{array}\!\right]
    \!\left[\begin{array}{X} 3/4\\ 1/2\\ 7/10\\ 4/5 \end{array}\!\right] + 
    \!\left[\begin{array}{CCCC} 3 & 2&-4&1\\1 & 0 & 2 & 2\\2 & 0 & -1 & 1\\0 & 1 & -2 & 3 \end{array}\!\right]
    \!\left[\begin{array}{C} 0\\ -2\\ -1\\ 0 \end{array}\!\right]
    \in \mathcal T\left(\left[\begin{array}{C} 0\\ -2\\ -1\\ 0 \end{array}\right],\{3,4\}\right)
\end{align*}

Furthermore, this point is also on the boundary of two tiles. Specifically, 
\begin{align*}
    \!\left[\begin{array}{X} -2\\1\\-1/2\\-1/2 \end{array}\!\right] &= 
    \!\left[\begin{array}{CCCC} 0&2&-4&0\\0 & 0 & 2 & 0\\-2 & 0 & 0 & -1\\0 & 0 & 0 & -3 \end{array}\!\right]
    \!\left[\begin{array}{X} 1/6\\ 0\\ 1/2\\ 1/6 \end{array}\!\right] + 
    \!\left[\begin{array}{CCCC} 3 & 2&-4&1\\1 & 0 & 2 & 2\\2 & 0 & -1 & 1\\0 & 1 & -2 & 3 \end{array}\!\right]
    \!\left[\begin{array}{C} 0\\ 0\\ 0\\0  \end{array}\!\right]\\
    \!\left[\begin{array}{X} -2\\1\\-1/2\\-1/2 \end{array}\!\right] &= 
    \!\left[\begin{array}{CCCC} 0&0&-4&1\\0 & 0 & 2 & 2\\-2 & 0 & 0 & 0\\0 & -1 & 0 & 0 \end{array}\!\right]
    \!\left[\begin{array}{X} 1/4\\ 1/2\\ 1/2\\ 0 \end{array}\!\right] + 
    \!\left[\begin{array}{CCCC} 3 & 2&-4&1\\1 & 0 & 2 & 2\\2 & 0 & -1 & 1\\0 & 1 & -2 & 3 \end{array}\!\right]
    \!\left[\begin{array}{C} 0\\ 0\\ 0\\ 0 \end{array}\!\right]
    \end{align*}

By Lemma~\ref{lem:halfopen}, it follows that \[(-2,1,-\frac12,-\frac12)^\top \in \mathcal T(\mbf 0 , \{2,3\}) \text{ if and only if }(S_{\{2,3\}}(M)^{-1}\mbf w)_2 >0.\] 
Similarly, \[(-2,1,-\frac12,-\frac12)^\top \in \mathcal T(\mbf 0 , \{3,4\})\text{ if and only if }(S_{\{3,4\}}(M)^{-1}\mbf w)_4 >0.\] 
After another calculation, one finds that for $\mbf w = (w_1,w_2,w_3,w_4)^\top$, we have 
\[(S_{\{2,3\}}(M)^{-1}\mbf w)_2 = \frac12 w_1 + w_2 \text{ and }(S_{\{3,4\}}(M)^{-1}\mbf w)_4 = \frac 15 w_1 + \frac 25 w_2.\] 
These entries are both positive if $w_1>2w_2$ and both negative if $w_1 < 2w_2$. 
Note that we cannot have $w_1 = 2w_2$, or else $\mbf w$ would not be sufficiently generic. 
In particular, \(\w'\) would be in the span of \(\mbf b_4\).

In conclusion, there are two possibilities. If $w_1 > 2w_2$, then $(-2,1,-\frac12,-\frac12)^\top$ is in three positive tiles and two negative tiles. 
Alternatively, if $w_1 < 2w_2$, then $(-2,1,-\frac12,-\frac12)^\top$ is in two positive tiles and one negative tile. 
In either case, this point is in exactly one more positive tile than negative tile, and thus satisfies Theorem~\ref{thm:mainthm}.
\end{example}

When one of  ${\mathbf T}^+(M)$ or ${\mathbf T}^-(M)$ is empty, Theorem~\ref{thm:mainthm} specializes to a result about more traditional tilings.  
We state only the version where ${\mathbf T}^-(M)$ is empty, but the same statement holds if ``non-negative'' is replaced with ``non-positive''. 

\begin{corollary}\label{cor:traditionaltiling}
{\cite[Corollary~9.2.8]{alexthesis}}
    If the sign of $\det(S_\sigma(M))$ is non-negative for each $\sigma \in \binom{[r+k]}{r}$, then
    \[\R^{r+k} = \bigsqcup_{\mbf z \in \mbb Z} \left( \bigsqcup_{\sigma \in \binom{[r+k]}{r}} \mathcal T(\mbf z, \sigma) \right).\] 
\end{corollary}
\begin{proof}
    This follows immediately after applying Lemma~\ref{lem:support} to Theorem~\ref{thm:mainthm}.
\end{proof}

\begin{remark}
    The conditions required on $M$ for Corollary~\ref{cor:traditionaltiling} to apply are discussed in~\cite[Section~6.7]{alexthesis}. 
    The original proof of the corollary relies on these properties, so we needed different methods to prove the more general Theorem~\ref{thm:mainthm}. 
    A special case of Corollary~\ref{cor:traditionaltiling} was used in~\cite{multijections} to define a family of \emph{multijections} between the \emph{sandpile group} and \emph{cellular spanning forests} for a large class of cell complexes. 
\end{remark}

\section{Average Weight of the Tiling}\label{sec:average}

In this short section, we show by determinant computations that the average value of $f(\mbf p)$ is $(-1)^{k}\sgn(\det(M))$. 
To state this idea precisely, we use integrals and ideas from introductory calculus. 

\begin{lemma}\label{lem:SandC}
For any $\sigma \in \binom{[r+k]}{r}$, the following equality holds:
\[\det(S_\sigma(M)) = \det({C}_{\sigma}(M)) \cdot \det(\oline{C}_{\what \sigma}(M)) \cdot \sgn(\sigma,\what\sigma).\]
\end{lemma}

The next lemma is a direct application of the Laplace determinant expansion formula. 
Note that the $(-1)^{k}$ term is included because the bottom $k$ rows of $M$ are given by $-\oline C_{[r+k]}(M)$ instead of $\oline C_{[r+k]}(M)$. 

\begin{lemma}[Multiple Row Laplace Expansion]\label{lem:laplace}
Using the notation above, we have the following chain of equalities. 
\[(-1)^{k} \det(M) = \sum_{\sigma \in \binom{[r+k]}{r}} \det(S_\sigma(M)) = \sum_{\sigma \in \binom{[r+k]}{r}} \det(C_\sigma(M))\det(\oline{C}_{\what \sigma}(M))\sgn(\sigma,\what \sigma).\] 
\end{lemma}

Now, we are ready for the main result of the section. 
\begin{theorem}\label{thm:integralf} Let $f$ be the function defined in Theorem~\ref{thm:mainthm}. Then, 
     \[\int_{\parp(M)} f(\mbf x) \mathrm d\mbf x = (-1)^{k} \det(M) .\]
\end{theorem}

\begin{proof}
First, note that for any $\mbf p \in \R^{r+k}$, there is precisely one $\mbf z \in \Z^{r+k}$ such that $\mbf p - M\mbf z \in \parp(M)  $. 
This is a direct consequence of Lemma~\ref{lem:basictiling}. 
Referring to Definition~\ref{def:tile}, for a point \(\mbf p \in \parp (S_{\sigma}(M)) \), there is exactly one choice of  $\mbf z \in \Z^{r+k}$ so that the intersection of \(\mathcal T(\mbf z,\sigma)\) and \(\parp(M)\) contains \(\mbf p\).
Therefore,  all the intersections of the various \(\mathcal T(\mbf z,\sigma)\) with \(\parp(M)\) contains equivalent points to a single copy, \(\mathcal T(\mbf 0,\sigma)\).
This gives us the following equality: 
\begin{equation*}\sum_{\mbf z \in \Z^{r+k}} \int_{\parp(M)}\supp_{\mathcal T(\mbf z,\sigma)}(\mbf x) \mathrm d \mbf x 
= \int_{\R^{r+k}}\supp_{\mathcal T(\mbf 0,\sigma)}(\mbf x) \mathrm d\mbf x 
\end{equation*}

Directly applying definitions gives
\begin{equation*}
\int_{\R^{r+k}}\supp_{\mathcal T(\mbf 0,\sigma)}(\mbf x) \mathrm d\mbf x 
= \int_{\R^{r+k}}\supp_{S_\sigma(M)}(\mbf x) \mathrm d\mbf x  
= |\det(S_\sigma(M))|.
\end{equation*}

Combining the definition of $f$, a carefully considered change to the order of summation and integration, the definitions of $\mathbf T^+(M)$ and $\mathbf T^-(M)$, and the above equalities, it follows that:

\begin{align*}\int_{\parp(M)} f(\mbf x) \mathrm d\mbf x &= \int_{\parp(M)} \sum_{T \in \mathbf T^+(M)} \supp_{T}(\mbf x) \mathrm d\mbf x - \int_{\parp(M)} \sum_{T \in \mathbf T^-(M)} \supp_{T}(\mbf x) \mathrm d\mbf x\\ &= \sum_{T \in \mathbf T^+(M)}\int_{\parp(M)}  \supp_{T}(\mbf x) \mathrm d\mbf x - \sum_{T \in \mathbf T^-(M)}\int_{\parp(M)}  \supp_{T}(\mbf x) \mathrm d\mbf x \\ 
&= \sum_{\sigma \in \binom{[r+k]}{r},~\det(S_\sigma(M))>0,~\mbf z \in \Z^{r+k}}\int_{\parp(M)}  \supp_{\mathcal T(\mbf z,\sigma)}(\mbf x) \mathrm d\mbf x - \\
&\hspace{3 cm}\sum_{\sigma \in \binom{[r+k]}{r},~\det(S_\sigma(M))<0,~\mbf z \in \Z^{r+k}}\int_{\parp(M)}  \supp_{\mathcal T(\mbf z,\sigma)}(\mbf x) \mathrm d\mbf x\\
&= \sum_{\sigma \in \binom{[r+k]}{r},~\det(S_\sigma(M))>0}|\det(S_\sigma(M))| - \sum_{\sigma \in \binom{[r+k]}{r},~\det(S_\sigma(M))<0}|\det(S_\sigma(M))|\\
&= \sum_{\sigma \in \binom{[r+k]}{r}}\det(S_\sigma(M)).\end{align*}
The result now follows directly from Lemma~\ref{lem:laplace}. 
\end{proof}

From here, we can directly compute the average value of $f$ over the domain $\parp(M)$.
\begin{corollary}\label{cor:avgvalue}
    Let $f$ be the function defined in Theorem~\ref{thm:mainthm}. Then, 
    \[\left(\int_{\parp(M)} f(\mbf x) \mathrm d\mbf x\right)/\left(\int_{\parp(M)} \supp_{\parp(M)} \mathrm d\mbf x\right) = (-1)^k \sgn(\det(M)).\]
\end{corollary}
It is immediate from the definition of $\mathbf T(M)$ that $f(\mbf p) = f(\mbf p + M\mbf z)$ for any $\mbf z \in \Z^{r+k}$. 
In particular, Corollary~\ref{cor:avgvalue} also holds over any domain that is a union of translates of \(\parp(M)\) by integer linear combinations of the column of $M$. 
Since these translates cover $\R^{r+k}$ by Lemma~\ref{lem:basictiling}, we say colloquially that Corollary~\ref{cor:avgvalue} implies that $(-1)^k \sgn(\det(M))$ is the ``average value'' of $f$ over $\R^{r+k}$. 
In particular, once we show in Section~\ref{sec:together} that $f$ is constant, this will imply that its value is $(-1)^k \sgn(\det(M))$.

\section{Crossing Boundaries of Tiles}\label{sec:crossing}

The goal of this section is to prove the value of the function $f$ from Theorem~\ref{thm:mainthm} does not change when leaving one tile and entering another. 
In Section~\ref{sec:together}, we will use this idea to show that $f$ is constant over all of $\R^{r+k}$.
Most of this section is working towards a sub-goal, which is to pair up collections of facets that have opposite signs in a sense.
This pairing is illustrated in Figure~\ref{fig:projF}.

Recall that $f$ expresses the sum of indicator functions of a collection of half-open parallelepipeds. 
Each of these indicator functions is constant on all of $\R^{r+k}$, except on the boundary of the associated tile. 
Thus, in order to prove that $f$ is constant, it is useful to consider these boundary points, which lie in the facets of the tile.

\subsection{Initial Definitions and Lemmas} 
First, we give names for the \emph{facets} of a half-open parallelepiped $\Pi(N)$ (i.e., the maximal faces of the closure of $\Pi(N)$). 
We only give a definition for the facets of parallelepipeds under consideration in this work (namely elements of $\mbf T(M)$), but the definition is easily generalized.

\begin{definition}\label{def:facets}
Fix $\mbf z \in \Z^{r+k}$ and $\sigma \in \binom{[r+k]}{r}$. There are \(2(r+k)\) facets of the tile \( \mathcal T(\mbf z, \sigma) \), which come in pairs.
We first define the ``lower'' facet of each pair, then the ``upper'' facet based upon the ``lower''.
For $j \in [r+k]$, let
\begin{align*}
    \mathcal F( \mbf z, \sigma, j, 0)  :&= \sum_{i\in [r+k]\setminus \{j\}} \left\{ x_i S_\sigma(M)_i : \begin{array}{l}0 < x_i \leq 1 \text{ if $(S_\sigma(M)^{-1}\w)_i < 0$}\\0 \leq x_i<  1  \text{ if $(S_\sigma(M)^{-1}\w)_i > 0$}\end{array}\right\}+ M\mbf z,\\
    \mathcal F( \mbf z, \sigma, j, 1) :&=  \mathcal F( \mbf z, \sigma, j, 0) + S_\sigma(M)_j.
\end{align*} 
We conflate these two definitions into a single parametrization, \(\mathcal F( \mbf z, \sigma, j, s) \). 
Characterizing the parameters, \(\mbf z \in \Z^{r+k}\) is a translational parameter, \(\sigma \in \binom{[r+k]}{r}\) determines the fragment, \(j \in [r+k]\) indexes which facet of  \( \mathcal T(\mbf z, \sigma) \) is considered and indicates which vector of \(S_\sigma(M)\) is excluded, and \(s \in \{0,1\}\) is the choice between ``upper'' and ``lower'' facet.
When $S_\sigma(M)$ is not invertible, let $ \mathcal F( \mbf z, \sigma, j, s) := \emptyset$.
\end{definition}

While definitionally, \(\mathcal F( \mbf z, \sigma, j, s) \) is a subset of \(\mathbb{R}^{r+k}\), we conflate this notion with another, where \(\mathcal F( \mbf z, \sigma, j, s) \) is a single discrete object, the combinatorial facet of some parallelepiped.
We switch between these perspectives, depending on the context.

In practice, we find that it was more useful to work with a slight variant of $\mathcal F( \mbf z, \sigma, j, s)$, which we define below. 

\begin{definition}\label{def:tildefacet} 
Fix $\sigma \in \binom{[r+k]}{r}$. 
For $j \in [r+k]$ and \(s \in \{0,1\}\), let

\[ \mathcal{\widetilde F} ( \mbf z, \sigma, j, s) := \begin{cases} \mathcal F( \mbf z- s \mbf{e}_j, \sigma, j, s) & \text{ if $j \in \sigma$,}\\
\mathcal F( \mbf z + s \mbf{e}_j, \sigma, j, s) & \text{ otherwise,}\end{cases} 
\]
where \(\mbf{e}_j\) is the \(j^{th}\) standard basis vector. Note that $\mathcal{\widetilde F} ( \mbf z, \sigma, j, 0) = \mathcal{F} ( \mbf z, \sigma, j, 0)$ for any choice of $\mbf z$, $\sigma$, and $j$. 
\end{definition}

Definition~\ref{def:tildefacet} is a bit less natural than Definition~\ref{def:facets} since $\mathcal F( \mbf z, \sigma, j, s)$ is always a facet of $\mathcal T(\mbf z,\sigma)$, while  $\mathcal{\widetilde F} ( \mbf z, \sigma, j, s)$ is a facet of $\mathcal T(\mbf z,\sigma)$ if $s = 0$ and a facet of $\mathcal T(\mbf z \pm \mbf{e}_j,\sigma)$ if $s=1$. 
However, this groups the facets into more convenient collections.  
We will show in Section~\ref{sec:groupingfacets} the reasons we made this change, and the convenience that arises.

Recall that each fragment matrix $S_\sigma(M)$ is essentially a block matrix, where the first \(r\) coordinates of each vector are either those from \(M\) or \(0\), and the last \(k\) coordinates are either the negatives of \(M\) or \(0\). 
Because of this, it will be useful for us to consider the following projection maps. 

\begin{definition}\label{def:proj}
We write $\proj$ for the map from $\R^{r+k} \to \R^r$ which projects into the first $r$ coordinates and $\nproj$ for the map from $\R^{r+k} \to \R^k$ which projects to the last $k$ coordinates. In particular, for any $\mbf p \in \R^{r+k}$, we have
\[\mbf p = \begin{bmatrix} \proj(\mbf p) \\ \nproj(\mbf p) \end{bmatrix}.\] 
Recall that we define $\w' := \proj(\w)$ and $\w'':= \nproj(\w)$.
\end{definition}

The following lemma is immediate from the structure of $M$ and $S_\sigma(M)$. 
\begin{lemma}\label{lem:MSproj}
    Let $\sigma \in \binom{[r+k]}{r}$ and $j \in [r+k]$. Then $\proj(M_j) = {\mbf c}_j$, $\nproj(M_j) = -\oline{\mbf c}_j$, 
    \[ \proj(S_\sigma(M)_j) = \begin{cases} {\mbf c}_j &\text{if $j \in \sigma$,}\\ \mbf 0 & \text{otherwise,}\end{cases} \hspace{ .5 cm} \text{and} \hspace{ .5 cm}\nproj(S_\sigma(M)_j) = \begin{cases} \mbf 0 &\text{if $j \in \sigma$,}\\ \oline{\mbf c}_j & \text{otherwise.}\end{cases}\]    
\end{lemma}

The following two lemmas allow us to describe any facet $\widetilde{\mathcal F}(\mbf z, \sigma, j ,s)$ in terms of the columns of $C_\sigma(M)$ and $\oline{C}_{\what\sigma}(M)$.

\begin{lemma}\label{lem:proj} 
Let $\sigma \in \binom{[r+k]}{r}$ and $j \in [r+k]$ such that $S_\sigma(M)$ is invertible. We have the following equalities, 
\[\proj( \widetilde{\mathcal F}( \mbf z, \sigma, j, 0) ) = \sum_{i\in \sigma\setminus j} \left\{ x_i \mbf c_i: \begin{array}{l}0 < x_i \leq 1 \text{ if $(C_\sigma(M)^{-1}\w')_{\sigma^{-1}(i)} < 0$}\\0 \leq x_i<  1  \text{ if $(C_\sigma(M)^{-1}\w')_{\sigma^{-1}(i)} > 0$}\end{array} \right\} + \proj(M\mbf z) .\]

\[\nproj( \widetilde{\mathcal F}( \mbf z, \sigma, j, 0) ) = \sum_{i\in \what\sigma\setminus j} \left\{ x_i \oline{\mbf c}_{i}: \begin{array}{l}0 < x_i \leq 1 \text{ if $(\oline{C}_{\what\sigma}(M)^{-1}\w'')_{\what\sigma^{-1}(i)} < 0$}\\0 \leq x_i<  1  \text{ if $(\oline{C}_{\what\sigma}(M)^{-1}\w'')_{\what\sigma^{-1}(i)} > 0$}\end{array} \right\}+ \nproj(M\mbf z).\]
\end{lemma}
\begin{proof}
    First, note that $\widetilde{\mathcal F}( \mbf z, \sigma, j, 0) = {\mathcal F}( \mbf z, \sigma, j, 0)$, so we can focus on Definition~\ref{def:facets}. We claim that for any $i\in [r+k]$, the following equalities hold:
    \begin{align*}(S_{\sigma}(M)^{-1}\mbf w)_i &= \begin{cases} (C_\sigma(M)^{-1}\mbf w')_{\sigma^{-1}(i)} & \text{ if $i \in \sigma$, and}\\ (\oline{C}_{\what\sigma}(M)^{-1}\mbf w'')_{\what\sigma^{-1}(i)} & \text{ if $i \in \what\sigma$.}\end{cases}\\
    \proj(x_iS_\sigma(M)_i) &= \begin{cases} x_i \mbf c_i & \text{ if $i \in \sigma$, and}\\ \mbf 0 & \text{ if $i \in \what\sigma$.}\end{cases}\\
    \nproj(x_iS_\sigma(M)_i) &= \begin{cases} \mbf 0 & \text{ if $i \in \sigma$, and}\\  x_i \oline{\mbf c}_i  & \text{ if $i \in \what\sigma$.}\end{cases}\end{align*}
    All three parts of this claim follow from the block structure of $S_\sigma(M)$, and the bottom two equalities are also simple corollaries of Lemma~\ref{lem:MSproj}. 
    From here, the result follows from Definition~\ref{def:facets}. 
\end{proof}

\begin{lemma}\label{lem:proj2} 
Let $\sigma \in \binom{[r+k]}{r}$ and $j \in [r+k]$.

\[ \proj(\mathcal{\widetilde F} ( \mbf z, \sigma, j, 1)) = \begin{cases} \proj(\mathcal {\widetilde F}( \mbf z, \sigma, j, 0)) & \text{if $j \in \sigma$,} \\ \proj(\mathcal {\widetilde F}( \mbf z, \sigma, j, 0)) + \mbf c_j & \text{otherwise.}\end{cases}\]

\[\nproj(\mathcal{\widetilde F} ( \mbf z, \sigma, j, 1) ) = \begin{cases} \nproj(\mathcal {\widetilde F}( \mbf z, \sigma, j, 0)) + \oline{\mbf c}_j& \text{if $j \in \sigma$,} \\ 
\nproj(\mathcal {\widetilde F}( \mbf z, \sigma, j, 0))& \text{otherwise.}\end{cases}\] 
\end{lemma}
\begin{proof}
    First, suppose that $j \in \sigma$. Then, by combining Definitions~\ref{def:facets} and \ref{def:tildefacet}, and applying Lemma~\ref{lem:MSproj}, we find that 
    \begin{align*} \proj(\mathcal{\widetilde F} ( \mbf z, \sigma, j, 1)) 
    &= \proj(\mathcal F( \mbf z-  \mbf{e}_j, \sigma, j, 0) + S_\sigma(M)_j)\\ 
    &= \proj(\mathcal {\widetilde F} ( \mbf z, \sigma, j, 0)- M\mbf e_j + S_\sigma(M)_j )\\ 
    &= \proj(\mathcal {\widetilde F} ( \mbf z, \sigma, j, 0))- \proj(M_j) + \proj(S_\sigma(M)_j) \\
    &= \proj(\mathcal {\widetilde F} ( \mbf z, \sigma, j, 0))- \mbf c_j + \mbf c_j \\ &= \proj(\mathcal {\widetilde F} ( \mbf z, \sigma, j, 0)).\end{align*}

    Similarly, 
    \begin{align*} \nproj(\mathcal{\widetilde F} ( \mbf z, \sigma, j, 1)) 
    &= \nproj(\mathcal {\widetilde F} ( \mbf z, \sigma, j, 0))- \nproj(M_j) + \nproj(S_\sigma(M)_j) \\
    &= \nproj(\mathcal {\widetilde F} ( \mbf z, \sigma, j, 0))- (-\oline{\mbf c}_j) + \mbf 0 \\ 
    &= \nproj(\mathcal {\widetilde F} ( \mbf z, \sigma, j, 0)) + \oline{\mbf c}_j.\end{align*}
    
    The case where $j \not\in \sigma$ is analogous. 
\end{proof}

\subsection{Finding Structure by Grouping Facets} \label{sec:groupingfacets}

We were not able to prove Theorem~\ref{thm:mainthm} by considering individual facets. 
Nevertheless, when collections of facets are grouped in a particular way, a useful structure becomes apparent. 
In particular, one can use Lemmas~\ref{lem:proj} and~\ref{lem:proj2} to show that each collection defined below is made up of facets contained within the same hyperplane. 

\begin{definition}\label{def:facetcollections}
    Fix $\mbf z \in \Z^{r+k}$ and $\tau \in \binom{[r+k]}{r-1}$. We define the collection of facets
\[\Facetset{~}{\mbf z, \tau} := \biguplus_{j \in \what \tau} 
\left( \mathcal{\widetilde F} ( \mbf z, \tau \cup j, j, 0)  
\uplus  \mathcal{\widetilde F} ( \mbf z, \tau \cup j, j, 1) \right) .\]
Similarly, fix $\mbf z \in \Z^{r+k}$ and $\gamma \in \binom{[r+k]}{r+1}$. We define the collection of facets
\[\FacetsetO{~}{\mbf z, \gamma} := \biguplus_{j \in \gamma} 
\left( \mathcal{\widetilde F}  ( \mbf z, \gamma \setminus j, j, 0)  
\uplus  \mathcal{\widetilde F} ( \mbf z, \gamma \setminus j, j, 1) \right).\]
\end{definition}

The following proposition and corollary show that the collections of \(\Facetset{~}{\mbf z, \tau}\) and \(\FacetsetO{~}{\mbf z, \gamma}\) partition the set of facets.

\begin{prop}\label{prop:uniquetile} 
Let \(\mathcal F( \mbf z, \sigma, j, s) \) be a facet of the tile \(\mathcal T(\mbf z, \sigma) \in \mathbf T(M)\).

If \(j \in \sigma\), then  \[ \mathcal{F} ( \mbf z, \sigma, j, s) = \mathcal{\widetilde F}( \mbf z + s \mbf {e}_j, \sigma, j, s) \in  \Facetset{~}{ \mbf z + s \mbf {e}_j, \sigma \setminus j} .\]

If \(j \not \in \sigma\), then  \[ \mathcal{F} ( \mbf z, \sigma, j, s) =\mathcal{\widetilde F}( \mbf z - s \mbf {e}_j, \sigma, j, s) \in \FacetsetO{~}{\mbf z - s \mbf {e}_j, \sigma \cup j}.\]

\end{prop}

\begin{proof}
The first equality is immediate from Definition~\ref{def:tildefacet}, while the inclusion follows from Definition~\ref{def:facetcollections}.
\end{proof}

\begin{corollary}\label{cor:allFacets} Given the definitions above, we find the following equality: 
\label{eq:facetgroups}\[ \biguplus_{\tau \in \binom{[r+k]}{r-1},~ \mbf z \in \mathbb Z^{r+k}} \Facetset{~}{\mbf z,\tau}
\hspace{10pt} \mathlarger{\biguplus} \hspace{10pt}
\biguplus_{\gamma \in \binom{[r+k]}{r+1},~ \mbf z \in \mathbb Z^{r+k}} \FacetsetO{~}{\mbf z,\gamma} =\] 
\[ \biguplus_{\mbf z \in \mathbb Z^{r+k},~\sigma \in \binom{[r+k]}{r},~j \in [r+k],~s \in \{0,1\}}\mathcal{F} ( \mbf z, \sigma, j, s).\]
\end{corollary}
\begin{proof}
First note that for any fixed $\mbf z \in \Z^{r+k}$, we get the following chain of equalities. 

\begin{align*}\biguplus_{\tau \in \binom{[r+k]}{r-1}} \Facetset{~}{\mbf z,\tau} &= \biguplus_{\tau \in \binom{[r+k]}{r-1},~j \in \what \tau} 
\left( \mathcal{\widetilde F} ( \mbf z, \tau \cup j, j, 0)  
\uplus  \mathcal{\widetilde F} ( \mbf z, \tau \cup j, j, 1) \right)\\
&= \biguplus_{\sigma \in \binom{[r+k]}{r},~j \in \sigma} 
\left( \mathcal{\widetilde F} ( \mbf z, \sigma, j, 0)  
\uplus  \mathcal{\widetilde F} ( \mbf z, \sigma, j, 1) \right)\\
&= \biguplus_{\sigma \in \binom{[r+k]}{r},~j \in \sigma} 
\left( \mathcal{F} ( \mbf z, \sigma, j, 0)  
\uplus  \mathcal{F} ( \mbf z- {\mbf e}_j, \sigma, j, 1) \right).\\\end{align*}

Since $\mbf z \to \mbf z + {\mbf e}_j$ gives a bijection on $\mathbb Z^{r+k}$ for every $j \in [r+k]$, taking the union over every $\mbf z \in \Z^{r+k}$ gives the following equality: 

\begin{align*}\biguplus_{\tau \in \binom{[r+k]}{r-1},~ \mbf z \in \mathbb Z^{r+k}} \Facetset{~}{\mbf z,\tau} = \biguplus_{\sigma \in \binom{[r+k]}{r},~j \in \sigma,~ \mbf z \in \mathbb Z^{r+k}} 
\left( \mathcal{F} ( \mbf z, \sigma, j, 0)  
\uplus  \mathcal{F} ( \mbf z, \sigma, j, 1) \right).\end{align*} 
An analogous calculation shows that:

\begin{align*}\biguplus_{\gamma \in \binom{[r+k]}{r+1},~ \mbf z \in \mathbb Z^{r+k}} \FacetsetO{~}{\mbf z,\gamma} = \biguplus_{\sigma \in \binom{[r+k]}{r},~j \in \what\sigma,~ \mbf z \in \mathbb Z^{r+k}} 
\left( \mathcal{F} ( \mbf z, \sigma, j, 0)  
\uplus  \mathcal{F} ( \mbf z, \sigma, j, 1) \right).\end{align*} 
The result follows from combining these two equalities. 
\end{proof}

Corollary~\ref{cor:allFacets} enables us to focus on sets of the form \(\Facetset{~}{\mbf z,\tau}\) and \(\FacetsetO{~}{\mbf z,\gamma}\) , while still retaining the full character of the problem.

One important property of these sets is the following result.

\begin{prop}\label{prop:sameproj}
    Fix any $\tau \in \binom{[r+k]}{r-1}$ and $\mbf z \in \Z^{r+k}$. The relative interior of $\proj(F)$ is the same for all $F \in \Facetset{~}{\mbf z, \tau}$. 
    In particular, this region is given by the open $(r-1)$-dimensional parallelepiped 
    \[ \sum_{i \in \tau} \left \{ x_i\mbf c_i : 0 < x_i < 1\right\} + \proj(M\mbf z).\]
    Similarly, fix any $\gamma \in \binom{[r+k]}{r+1}$ and $\mbf z \in \Z^{r+k}$. The relative interior of $\nproj(F)$ is the same for all $F \in \FacetsetO{~}{\mbf z, \gamma}$. In particular, this region is given by the open $(k-1)$-dimensional parallelepiped 
    \[ \sum_{i \in \what\gamma} \left \{ x_i\mbf c_i : 0 < x_i < 1\right\} + \nproj(M\mbf z).\]
\end{prop}
\begin{proof}
    This follows from directly from Lemmas~\ref{lem:proj} and~\ref{lem:proj2}. Notice that the expression in Lemma~\ref{lem:proj} is simpler when we are only concerned with the relative interior of $\proj(F)$. 
\end{proof}

For the remainder of this section, we will focus primarily on the $\Facetset{~}{\mbf z, \tau}$ setting, but analogous statements about $\FacetsetO{~}{\mbf z, \gamma}$ hold as well.

\subsection{``Up'' Facets and ``Down'' Facets}\label{sec:updown}\hspace{1pt}

For this and the following subsection, we will look closer at the facets which make up $\Facetset{~}{\mbf z, \tau}$ for a specific choice of $\mbf z \in \Z^{r+k}$ and $\tau \in \binom{[r+k]}{r-1}$. We begin by defining two functions, $\wsgn$ and $\tsgn$, which both map from $\Facetset{~}{\mbf z,\tau} \to \pm 1$.

First, we define $\wsgn$, which keeps track of whether $F$ is ``below'' or ``above'' its associated tile in the \(\mbf w\) direction.

\begin{definition}\label{def:wsgn}
Suppose that $F = \mathcal {\widetilde F} ( \mbf z, \tau \cup j, j, s) \in \Facetset{~}{\mbf z,\tau}$ and choose $\mbf p$ to be an arbitrary point on the interior of $F$. 
Let $\mbf q = \mbf p- M(\mbf z - s\mbf e_j)$, which is a point on the boundary of $\parp(S_{\tau \cup j}(M))$. 
Then,
\[ \wsgn(F) := \begin{cases} 1 & \text{if $\mbf q + \epsilon \w \in \parp(S_{\tau \cup j}(M))$ for all sufficiently small $\epsilon > 0$,}\\ -1 &  \text{otherwise.}\end{cases}\]

\end{definition}

Next, we define $\tsgn$, which keeps track of whether $F$ is a facet of a ``positive tile'' or a ``negative tile''. 

\begin{definition}\label{def:tsgn}Suppose that $F = \mathcal {\widetilde F} ( \mbf z, \tau \cup j, j, s) \in \Facetset{~}{\mbf z,\tau}$. Then, 
\[\tsgn(F) := \begin{cases}
1  & \text{if $\det(S_{\tau \cup j}(M)) > 0$,}\\ 
-1 & \text{otherwise.} 
    \end{cases}\]
\end{definition} 

Recall that $F = \mathcal {\widetilde F} ( \mbf z, \tau \cup j, j, s)$ is a facet of the parallelepiped $T$, where  \[T = \mathcal T(\mbf z+s\mbf{e}_j,\tau\cup j) = \parp(S_{\tau \cup j}(M)) + M(\mbf z + s\mbf{e}_j).\] It follows from Definition~\ref{def:wsgn}, that a particle crossing $F$ in the direction $\w$ will enter $T$ if $\wsgn(F) = 1$ and will exit $T$ if $\wsgn(F) = -1$. 

Furthermore, recall that $T \in \mathbf T^+(M)$ if $\det(S_{\tau\cup j})(M) > 0$ and $T \in \mathbf T^-(M)$ if $\det(S_{\tau\cup j}(M)) < 0$. It follows from Definition~\ref{def:tsgn} that $T \in \mathbf T^+(M)$ if $\tsgn(F) = 1$ and $T \in \mathbf T^-(M)$ if $\tsgn(F) = -1$. 

Next, consider the product $\wsgn(F)\tsgn(F)$. If $\wsgn(F)\tsgn(F) = 1$, then a particle crossing $F$ if the direction $\w$ either \emph{enters} a \emph{positive} tile or \emph{exits} a \emph{negative} tile. Alternatively, if $\wsgn(F)\tsgn(F) = -1$, then the particle either \emph{enters} a \emph{negative} tile or \emph{exits} a \emph{positive} tile. 

Note that even through we only defined $\wsgn(F)$ and $\tsgn(F)$ for $F \in \Facetset{~}{\mbf z,\tau}$, analogous definitions hold for $F \in \FacetsetO{~}{\mbf z,\gamma}$. See the end of this section for a bit more discussion about this generalization.

Recall the function $f$ defined in Theorem~\ref{thm:mainthm}. The product $\wsgn(F)\tsgn(F)$ indicates the change in contribution to $f$ coming from $T$ when a particle crosses $F$ in the $\w$ direction. More precisely, we have the following result. 

\begin{lemma}\label{lem:contrib}
    Fix $\mbf p \in \R^{r+k}$ and $\epsilon>0$ such that the line segment between $\mbf p$ and $\mbf p + \epsilon\w$  intersected with the set of all facets of tiles in $\mbf T(M)$ forms a single point, $\mbf q$, which lies on the interior of this line segment. 
    Then,
    \[ f(\mbf p + \epsilon\w) - f(\mbf p) = \sum_{F \ni \mbf q} \wsgn(F)\tsgn(F),\]
    where $F \ni \mbf q$ indicates the set of all facets of tiles in $\mbf T(M)$ which contain the point $\mbf q$. 
\end{lemma}

With Lemma~\ref{lem:contrib} in mind, we partition each set $\Facetset{~}{\mbf z, \tau}$ into two subsets. 
\begin{definition}\label{def:FupFdown}
    Let
    \begin{align*}\Facetset{\uparrow}{\mbf z,\tau} &:= 
    \biguplus_{i \in \what\tau}\left\{
    F \in \Facetset{~}{\mbf z, \tau} : \wsgn(F)\tsgn(F) > 0\right\} \text{, and}\\
    \Facetset{\downarrow}{\mbf z,\tau} &:= 
    \biguplus_{i \in \what\tau}\left\{
    F \in \Facetset{~}{\mbf z, \tau} : \wsgn(F)\tsgn(F) < 0\right\}.
    \end{align*}
\end{definition}

In the next two subsections, we will prove Corollary~\ref{cor:DoubleIntersection} which shows that the facets in $\Facetset{\uparrow}{\mbf z,\tau}$ and the facets in $\Facetset{\downarrow}{\mbf z,\tau}$ each contain a collection of disjoint subsets of $\R^{r+k}$ whose union covers the same subset of $\R^{r+k}$ (except possibly a minor inconsistency at the boundary). 

In Section~\ref{sec:together}, we will combine this result with Lemma~\ref{lem:contrib} to show that $f$ is constant.

\subsection{A Few Linear Algebra Tools}

The next two subsections form the most technical part of the paper. In this section, we give some general linear algebra techniques. 

We first give two versions of the well-known Cramer's rule (see, e.g., \cite[Section~2.1.2]{Nonlinear}). 

\begin{lemma}[Cramer's Rule Version 1] \label{lem:Cramer1}

Let ${\mbf v}_1,\dots,{\mbf v}_{n+1},\mbf a$ be a collection of vectors in $\R^n$ such that  
\[\begin{bmatrix}
\vrule height 10pt depth 0pt width 0pt \vline & \vline &  & \vline \\
\mbf v_1 &\mbf v_2 & \hdots & \mbf v_n\\
\vrule height 10pt depth 0pt width 0pt \vline & \vline &  & \vline \\
\end{bmatrix} \begin{bmatrix}
\vrule height 10pt depth 0pt width 0pt \vline \\
\mbf a \\
\vrule height 10pt depth 0pt width 0pt \vline \\
\end{bmatrix} = \begin{bmatrix}
\vrule height 10pt depth 0pt width 0pt \vline \\
\mbf v_{n+1}\\
\vrule height 10pt depth 0pt width 0pt \vline \\
\end{bmatrix}.\]
For every $i \in [n]$, the $i^{th}$ entry of $\mbf a$, which we will denote $a_i$, is given by 
\begin{equation}\label{eq:cramer1}a_i = \det(\begin{bmatrix}
\vrule height 10pt depth 0pt width 0pt \vline & &\vline & \vline & \vline & & \vline \\
\mbf v_1 & \hdots & \mbf v_{i-1}&\mbf v_{n+1} &\mbf v_{i+1} & \hdots & \mbf v_n\\
\vrule height 10pt depth 0pt width 0pt \vline & &\vline & \vline & \vline &  & \vline \\
\end{bmatrix})/ \det(\begin{bmatrix}
\vrule height 10pt depth 0pt width 0pt \vline & \vline &  & \vline \\
\mbf v_1 &\mbf v_2 & \hdots & \mbf v_n\\
\vrule height 10pt depth 0pt width 0pt \vline & \vline &  & \vline \\
\end{bmatrix})\end{equation}

\end{lemma}

Notice that by rearranging the columns of the numerator of~\eqref{eq:cramer1}, we obtain the following expression for $a_i$. 
\begin{equation}\label{eq:cramer1alt}a_i = (-1)^{n-i}\det(\begin{bmatrix}
\vrule height 10pt depth 0pt width 0pt \vline & &\vline & \vline & & \vline \\
\mbf v_1 & \hdots & \mbf v_{i-1} &\mbf v_{i+1} & \hdots & \mbf v_{n+1}\\
\vrule height 10pt depth 0pt width 0pt \vline & &\vline & \vline &   & \vline \\
\end{bmatrix})/ \det(\begin{bmatrix}
\vrule height 10pt depth 0pt width 0pt \vline & \vline &  & \vline \\
\mbf v_1 &\mbf v_2 & \hdots & \mbf v_n\\
\vrule height 10pt depth 0pt width 0pt \vline & \vline &  & \vline \\
\end{bmatrix})\end{equation}

We will also use an alternate version of Cramer's rule, which can be obtained from Lemma~\ref{lem:Cramer1} through straightforward algebraic means. 
In particular, this second version follows from replacing the entries of $\mbf a$ in the matrix equation with the expressions given in~\eqref{eq:cramer1alt}, then expanding the product, multiplying the denominator, and bringing all of the terms to the same side. 

\begin{lemma}[Cramer's Rule Version 2] \label{lem:Cramer2}

Let ${\mbf v}_1,\dots,{\mbf v}_{n+1}$ be a collection of vectors in $\R^n$. Then, 
\[\sum_{i \in [n+1]}(-1)^i\det(\begin{bmatrix}
\vrule height 10pt depth 0pt width 0pt \vline & &\vline  & \vline & & \vline \\
\mbf v_1 & \hdots & \mbf v_{i-1}&\mbf v_{i+1} & \hdots & \mbf v_{n+1}\\
\vrule height 10pt depth 0pt width 0pt \vline & &\vline  & \vline &  & \vline \\
\end{bmatrix})\begin{bmatrix}
\vrule height 10pt depth 0pt width 0pt \vline \\
\mbf v_{i}\\
\vrule height 10pt depth 0pt width 0pt \vline \\
\end{bmatrix} = \mbf 0\]
\end{lemma}

Before returning to the context of the paper, we give one more technical result which follows from Cramer's rule. We will work with vectors $\mbf v_1,\dots,\mbf v_{k+1} \in \R^k$, and write $\mbf w''$ for our sufficiently generic vector. Let 

\[ V := \begin{bmatrix}
\vrule height 10pt depth 0pt width 0pt \vline & \vline &  & \vline \\
\mbf v_1 &\mbf v_2 & \hdots & \mbf v_{k+1}\\
\vrule height 10pt depth 0pt width 0pt \vline & \vline &  & \vline \\
\end{bmatrix}\]
be an $k \times (k+1)$ matrix of rank $k$. For each $i \in [k+1]$, let $V_{\what i}$ be the $k \times k$ matrix obtained by removing the $i^{th}$ column of $V$. Additionally, recall the zonotope $Z(V)$ which was defined in Definition~\ref{def:halfopen}. Note that since the columns of $V$ have rank $k$, the kernel of $V$ is one-dimensional. 

\begin{prop}\label{prop:doublecover}
Consider the definitions in the previous paragraph and let $\mbf h = (h_1,\dots,h_{n+1})^\top$ be any non-zero vector in the kernel of $V$. Then, 
\begin{equation}\label{eq:ztiling}
\bigsqcup_{i\in [n+1]}\left\{\begin{array}{ll}\parp(V_{\what i}) &\text{if $h_i<0$}\\\parp(V_{\what i}) + {\mbf v}_{i} &\text{if $h_i>0$}\end{array}\right\} =Z(V).\end{equation}
\end{prop}
\begin{proof}
    It is immediate that all of the parallelepipeds on the left are contained in the zonotope on the right. Thus, it suffices to show that for any $\mbf p \in Z(V)$, there is a unique $i \in \what\tau$ such that $h_i > 0$ and $\mbf p \in \parp(V_{\what i})$ or $h_i < 0$ and $\mbf p \in \parp(V_{\what i})+ {\mbf v}_{i}$. 

    Fix $\epsilon>0$, ensuring that this value is small enough that $\mbf p + \epsilon \mbf w''\in Z(V)$. By definition, there exists some \[\mbf x = (x_1,x_2,\dots,x_{k+1})^\top \in [0,1]^{k+1}\] such that $\mbf p + \epsilon \mbf w'' = V\mbf x=\sum_{i \in [k+1]} x_i{\mbf v}_{i}$. Furthermore, since $\mbf h$ generates the kernel of $V$, it follows that for every real number $\zeta$, we have
    \[\mbf p + \epsilon\mbf w''= V(\mbf x + \zeta \mbf h).\]
    In fact, the line given by $\mbf x + \zeta \mbf h$ is precisely the set of vectors which are mapped to $\mbf p + \epsilon\mbf w''$ by $V$. 

    Next, we consider the intersection of the line $\mbf x + \zeta \mbf h$ with the region $[0,1]^{k+1}$. This is a line segment containing the point $\mbf x$. Let $\mbf y = (y_1,\dots,y_{k+1})^\top$ be the point on this line segment where $\zeta$ is maximized.

    In order for $\mbf p + \epsilon \mbf w''$ to be in at least one of the parallelepipeds on the left of~\eqref{eq:ztiling}, it is necessary that for some $j \in \what\tau$, one of the following conditions hold:
    \begin{enumerate}
        \item $\mbf p+\epsilon\mbf w'' = V\mbf u$ for some $\mbf u= (u_1,\dots,u_{k+1}) \in [0,1]^{k+1}$ where $u_j =0$. Furthermore, $h_j < 0$.
        \item $\mbf p+\epsilon\mbf w'' = V\mbf u$ for some $\mbf u= (u_1,\dots,u_{k+1}) \in [0,1]^{k+1}$ where $u_j =1$. Furthermore, $h_j > 0$.
    \end{enumerate}

    For both cases, given any $\epsilon' >0$, we must have $\mbf u + \epsilon' \mbf h \not\in [0,1]^{k+1}$. In particular, the $j^{th}$ entry of this vector is less than $0$ in the first case and greater than $1$ in the second. The only vector $\mbf u \in (\mbf x + \zeta \mbf h)\cap[0,1]^{k+1}$ for which this condition holds is $\mbf y$, since this is where $\zeta$ is maximized. Thus, we can restrict our attention to this vector. 

    We have established that $\mbf p + \epsilon\mbf w''= V\mbf y$, where $\mbf y \in [0,1]^{k+1}$, and for some $j \in [k+1]$, we have either $ y_j = 0$ and $h_j < 0$ or $ y_j=1$ and $h_j > 0$. 

    Suppose that there is more than one $j$ such that $y_j \in \{0,1\}$. Then, we just replace $\epsilon$ with $\epsilon/2$ and restart the proof. It follows from the fact that $\mbf w''$ is sufficiently generic that this process will eventually terminate. 
    
    Let $\mbf y'$ be the vector $\mbf y$ after removing the $y_j$ entry. Then, it follows that
    \[\mbf p+\epsilon\mbf w'' = \begin{cases}V\mbf y'& \text{ if $h_i<0$, and}\\
    V\mbf y' + {\mbf {v}}_i& \text{ if $h_i > 0$.}\end{cases}\]
    
We also know that $\mbf y' \in [0,1]^k$, so this implies that $\mbf p \in \parp(V_{\what i})$ if $h_i <0$ and $\mbf p \in \parp(V_{\what i}) + \mbf v_i$ if $h_i >0$. Furthermore, $\mbf p$ cannot be in any other parallelepipeds on the left of~\eqref{eq:ztiling} by the condition that $y_i \not\in \{0,1\}$ for $i \not= j$.  
\end{proof}

\begin{remark}
    Proposition~\ref{prop:doublecover} can be thought of as a special case of~\cite[Proposition 3.2.1]{BBY}, but it is presented here in a self-contained manner. 
\end{remark}

\subsection{A Pairing Among Facets}
With these tools in mind, we now return to our goal of proving Proposition~\ref{prop:doublecover}.

Consider some $\sigma \in \binom{[r+k]}{r}$ such that $S_\sigma(M)$ is invertible. 
We will write \begin{equation}\label{eq:lambda}\bm{\lambda}^\sigma = (\lambda^\sigma_1,\lambda^\sigma_2,\dots,\lambda^\sigma_{r+k})^\top := S_\sigma(M)^{-1}\w.\end{equation}
We also write $\bm{\lambda}^\sigma_\sigma$ and $\bm{\lambda}^\sigma_{\what \sigma}$ for the restriction of $\bm{\lambda}^\sigma$ to entries in $\sigma$ or $\what \sigma$ respectively. 

We can use the block structure of $S_\sigma(M)$ to give expressions for $\w'$ and $\w''$. 

\begin{lemma}\label{lem:restrictedLambda}
    Given $\bm{\lambda}^\sigma_\sigma$ and $\bm{\lambda}^\sigma_{\what \sigma}$ as defined in \eqref{eq:lambda}, $\w' = \proj(\w)$, and $\w'' = \nproj(\w)$. Then, we also have
    \[ C_\sigma(M)\bm{\lambda}^\sigma_\sigma = \w' \hspace{1cm}\text{ and } \hspace{1cm}\oline{C}_{\what \sigma}(M)\bm{\lambda}^\sigma_{\what \sigma} = \w''.\]
\end{lemma}

Recall the function $\wsgn$ defined in Definition~\ref{def:wsgn}. The vector $\bm{\lambda}$ offers an alternate expression for $\wsgn(F)$. 
\begin{lemma}\label{lem:wsgn}
    Fix $\tau \in \binom{[r+k]}{r-1}$ and consider $F \in  \Facetset{~}{\mbf z, \tau}$. Let $\lambda_j^{\tau \cup j}$ be as defined in~\eqref{eq:lambda}. 
    If $F = \mathcal {\widetilde F} ( \mbf z, \tau \cup j, j, s)$, then 
        \[\wsgn(F) = (-1)^s\sgn(\lambda_j^{\tau \cup j}).\]
\end{lemma}
\begin{proof}
Let $\mbf p$ be a point on the interior of $\mathcal {\widetilde F} ( \mbf z, \tau \cup j, j, s)$ and set $\mbf q = \mbf p - M(\mbf z + s \mbf e_j)$. Next, let \[\mbf x  = (x_1,\dots,x_n)^\top:= S_{\tau \cup j}(M)^{-1}\mbf q.\]
    By Definitions~\ref{def:facets} and~\ref{def:tildefacet}, we must have $0 < x_i < 1$ for all $i \in [r+k]\setminus j$ and $x_j = s$. 

Now, consider Definition~\ref{def:wsgn}. For a fixed $\epsilon >0$, we have 
    \[\mbf q + \epsilon \mbf w = S_{\tau \cup j}(M)\mbf x + \epsilon S_{\tau \cup j}(M) \bm {\lambda}^{\tau \cup j} = S_{\tau \cup j}(M)(\mbf x + \epsilon \bm {\lambda}^{\tau \cup j}).\]
It follows that $\mbf q + \epsilon \mbf w \in \parp(S_{\tau \cup j}(M))$ if and only if $0 < x_i + \epsilon {\lambda}^{\tau \cup j}_i <1$ for all $i \in [r+k]$. For $i \not=j$, this is always true when $\epsilon$ is sufficiently small. Thus, we just need to consider $x_j + \epsilon {\lambda}^{\tau \cup j}_i = s + \epsilon {\lambda}^{\tau \cup j}_i$. This sum is between $0$ and $1$ for sufficiently small $\epsilon>0$ if and only if $(-1)^s{\lambda}^{\tau \cup j}_i > 0$. 

\end{proof}

\begin{corollary}\label{cor:updownalt}
Fix $\mbf z \in \Z^{r+k}$ and $\tau \in \binom{[r+k]}{r-1}$, and for $j \in \what \tau$, let $\lambda_j^{\tau \cup j}$ be as defined above. Then, we have the following alternate version of Definition~\ref{def:FupFdown}:
    \begin{align*}\Facetset{\uparrow}{\mbf z,\tau} &= 
   \biguplus_{j \in \what \tau}\left\{\begin{array}{ll}
    \mathcal{\widetilde F} ( \mbf z, \tau \cup j, j, 0) &\text{if  $\lambda_{j}^{\tau \cup j} \det(S_{\tau \cup j}(M)) > 0$,}\\
    \mathcal{\widetilde F} ( \mbf z, \tau \cup j, j, 1) &\text{otherwise,}\end{array}\right\}\\
    \Facetset{\downarrow}{\mbf z,\tau} &= \biguplus_{j \in \what \tau}\left\{\begin{array}{ll}
    \mathcal{\widetilde F} ( \mbf z, \tau \cup j, j, 0) &\text{if  $\lambda_{j}^{\tau \cup j} \det(S_{\tau \cup j}(M)) < 0$,}\\
    \mathcal{\widetilde F} ( \mbf z, \tau \cup j, j, 1) &\text{otherwise.}\end{array}\right\}\\
    \end{align*}
\end{corollary}
\begin{proof}
    This follows immediately from Definition~\ref{def:FupFdown} after applying Definition~\ref{def:tsgn} and Lemma~\ref{lem:wsgn}.
\end{proof}

Next, we consider the projections of these facets to the last $k$ coordinates. For this result, we work under the same assumptions as those given before Proposition~\ref{prop:doublecover}. In particular, our parallelepipeds and zonotope are in $\R^k$ and the sufficiently generic vector is $\mbf w''$. 

\begin{corollary}\label{cor:updownproj}
Fix $\mbf z \in \Z^{r+k}$ and $\tau \in \binom{[r+k]}{r-1}$, and for $j \in \what \tau$, let $\lambda_j^{\tau \cup j}$ be as defined above. Recall that $\nproj(F)$ is the projection of $F$ to the last $k$ coordinates. We have the following two equalities:
    \begin{align*}\biguplus_{F \in \Facetset{\uparrow}{\mbf z,\tau}} \nproj(F)- \nproj(M\mbf z) &= 
   \biguplus_{j \in \what \tau}\left\{\begin{array}{ll}
    \parp(\oline{C}_{\what\tau\setminus j}(M)) &\text{if  $\lambda_{j}^{\tau \cup j} \det(S_{\tau \cup j}(M)) > 0$,}\\
    \parp(\oline{C}_{\what\tau\setminus j}(M)) + \oline{\mbf c}_j &\text{otherwise,}\end{array}\right\}\\
   \biguplus_{F \in \Facetset{\downarrow}{\mbf z,\tau}} \nproj(F) - \nproj(M\mbf z) &= 
   \biguplus_{j \in \what \tau}\left\{\begin{array}{ll}
    \parp(\oline{C}_{\what\tau\setminus j}(M)) &\text{if  $\lambda_{j}^{\tau \cup j} \det(S_{\tau \cup j}(M)) < 0$,}\\
    \parp(\oline{C}_{\what\tau\setminus j}(M)) + \oline{\mbf c}_j &\text{otherwise.}\end{array}\right\}\\
    \end{align*}
\end{corollary}
\begin{proof}
    The result follows after applying Lemmas~\ref{lem:proj} and~\ref{lem:proj2} to Corollary~\ref{cor:updownalt}. Note that when $j \in \sigma$, the second expression in Lemma~\ref{lem:proj} is equivalent to 
    \[\nproj( \widetilde{\mathcal F}( \mbf z, \sigma, j, 0) ) =\oline{C}_{\what\sigma}(M) + \nproj(M\mbf z).\]
\end{proof}
Notice that the right side of Corollary~\ref{cor:updownproj} looks similar to the left side of Proposition~\ref{prop:doublecover}. Now, we will explore the product $\lambda_{j}^{\tau \cup j} \det(S_{\tau \cup j}(M))$ in order to eventually apply this proposition. 

We can find an alternate expression for the value of $\lambda_j^{\tau \cup j}$ using Cramer's rule.
\begin{lemma}\label{lem:lambda}The following equality holds for any $\tau \in \binom{[r+k]}{r-1}$ and $j \in \what \tau$.  
    \[\lambda_j^{\tau \cup j} = 
    \frac{\det(\begin{bmatrix}C_\tau(M) \vert \w'\end{bmatrix})}{\det(C_{\tau\cup j}(M))} \cdot \frac{\sgn(\tau,j,\what\tau \setminus j)}{\sgn{(\tau\cup j, \what \tau \setminus j)}}.\] 
\end{lemma}

\begin{proof}
    Begin with the first expression in  Lemma~\ref{lem:restrictedLambda} and apply Cramer's rule (Lemma~\ref{lem:Cramer1}). Specifically,  we use the version presented in~\eqref{eq:cramer1alt}. Since entries indexed by elements of $\what{\tau}\setminus j$ are ignored, the result follows. 
\end{proof}

\begin{prop}\label{prop:h}
For $\tau \in \binom{[r+k]}{r-1}$, let $\mbf {h}= (h_{\what{\tau}(1)}, \dots, h_{\what\tau(k+1)})^\top$ be defined by 
\[ h_{j} := \lambda_{j}^{\tau \cup j} \det(S_{\tau \cup j}(M))\] 
for every $j \in \what\tau$. Then, 
\[\mbf {h} \in \ker(\oline{C}_{\what{\tau}}(M)).\]
\end{prop}
\begin{proof}
    We use Lemmas~\ref{lem:SandC} and~\ref{lem:lambda} to find the following sequence of equalities.
    
    \begin{align*}
    \sum_{j \in \what \tau} h_{j}\oline{\mbf c}_{j} &= \\
    \sum_{j\in\what\tau} \lambda_{j}^{\tau \cup j} \det(S_{\tau \cup j}(M))\oline{\mbf c}_{j} &=\\ 
    \sum_{j \in\what\tau} \frac{\det(\begin{bmatrix}C_\tau(M) \vert \w'\end{bmatrix})}{\det(C_{\tau\cup j}(M))} \frac{\sgn(\tau,j,\what\tau \setminus j)}{\sgn{(\tau\cup j, \what \tau \setminus j)}} \hspace{2.5cm}& \\
    \cdot \det(C_{\tau \cup j}(M))\det(\oline C_{\what\tau \setminus j}(M)) \sgn(\tau \cup j,\what\tau\setminus j)\oline{\mbf c}_{j} &=\\
    \det(\begin{bmatrix}C_\tau(M) \vert \w'\end{bmatrix})\sum_{j\in\what\tau} \sgn(\tau,j,\what\tau \setminus j) \det(\oline C_{\what\tau \setminus j}(M)) \oline{\mbf c}_{j}
    &=\\
    (-1)^{k(r+1)-1}\det(\begin{bmatrix}C_\tau(M) \vert \w'\end{bmatrix})\sum_{j\in\what\tau} \sgn(\what\tau \setminus j,j,\tau) \det(\oline C_{\what\tau \setminus j}(M)) \oline{\mbf c}_{j} &= \mbf 0.\end{align*}
    
    The first two equalities are direct substitution of previously defined values. The third equality involves cancellation of a pair of sign terms, a pair of determinant terms, and factoring out a determinant which is independent of \(j\). The fourth equality is a simple change of sign. The final equality holds from Lemma~\ref{lem:Cramer2}.
  The sum is $\mbf 0$, so the coefficients outside the sum are irrelevant.\end{proof}

\begin{figure}
    \centering
    \begin{tabular}{cc}
\begin{tikzpicture}[scale=1.4]
\scriptsize
\node[coordinate] (b1) at (-2, 0) {};
\node[coordinate] (b3) at (1, 2) {};
\node[coordinate] (b4) at (-1, 3) {};
\draw ($(b4)$ ) -- ($(b4) + (b3)$ );
\draw[dotted] ($(b4) + (b3)$ ) -- ($(b4) + (b1) + (b3)$ );
\draw ($(b4) + (b1) + (b3)$ ) -- ($(b4) + (b1)$ ) -- ($(b4)$ ); 
\node at ($(b4) + 0.5*(b3) + 0.5*(b1)$ ) {$\widetilde{\mathcal F}( \mbf 0, \{2,4\}, 4, 1)$};
\draw (0,0) -- (b1) -- ($(b1) + (b4)$ ) -- (b4) -- cycle; 
\node at ($0.5*(b1) + 0.5*(b4)$ ) {$\widetilde{\mathcal F}( \mbf 0, \{2,3\}, 3, 0)$};
\draw[dotted] (0,0) -- (b3) -- ($(b3) + (b4)$ );
\draw ($(b3) + (b4)$ ) -- (b4) -- (0,0); 
\node at ($0.5*(b3) + 0.5*(b4)$ ) {$\widetilde{\mathcal F}( \mbf 0, \{1,2\}, 1, 0)$};

\draw[->] (0,0) -- (1,1);
\node at (1.2,1) {$\mbf w''$};
\node[circle, fill=black, inner sep=2pt] (0) at (0,0) {};
\node at (0.2,0) {$\mbf 0$};

\end{tikzpicture}

&
\begin{tikzpicture}[scale=1.4]
\scriptsize
\node[coordinate] (b1) at (-2, 0) {};
\node[coordinate] (b3) at (1, 2) {};
\node[coordinate] (b4) at (-1, 3) {};
\draw (0,0) -- (b1) -- ($(b1) + (b3)$ ) -- (b3);
\draw[dotted] (b3) -- (0,0); 
\node at ($0.5*(b3) + 0.5*(b1)$ ) {$\widetilde{\mathcal F}( \mbf 0, \{2,4\}, 4, 0)$};
\draw[dotted] ($(b3)$ ) -- ($(b4) + (b3)$ ) -- ($(b4) + (b1) + (b3)$ );
\draw ($(b4) + (b1) + (b3)$ ) -- ($(b3) + (b1)$ ) -- (b3); 
\node at ($(b3) + 0.5*(b4) + 0.5*(b1)$ ) {$\widetilde{\mathcal F}( \mbf 0, \{2,3\}, 3, 1)$};
\draw ($(b1)$ ) -- ($(b4) + (b1)$ ) -- ($(b4) + (b1) + (b3)$ ) -- ($(b3) + (b1)$ ) -- cycle; 
\node at ($(b1) + 0.5*(b3) + 0.5*(b4)$ ) {$\widetilde{\mathcal F}( \mbf 0, \{1,2\}, 1, 1)$};

\draw[->] (0,0) -- (1,1);
\node at (1.2,1) {$\mbf w''$};
\node[circle, fill=black, inner sep=2pt] (0) at (0,0) {};
\node at (0.2,0) {$\mbf 0$};
\end{tikzpicture} \\
 \(\Facetset{\uparrow}{\mbf 0, \{2\} } \) & \(\Facetset{\downarrow}{\mbf 0, \{2\} } \)
    \end{tabular}
    
    \caption{On the left (resp. right) are the projections \(\nproj (F)\) for each facet \(F\) in \(\Facetset{\uparrow}{\mbf 0, \{2\} } \) (resp. \(\Facetset{\downarrow}{\mbf 0, \{2\} }\)). It follows from Theorem~\ref{thm:sameproj} that the union of the projections in either set corresponds to the zonotope $Z(\oline C_{\{1,3,4\}}(M))$.}
    \label{fig:projF}
\end{figure}

Now we reach the final results of the section, which are a culmination of all of our previous work. 

\begin{theorem}\label{thm:sameproj}
    Fix $\mbf z \in \Z^{r+k}$ and $\tau \in \binom{[r+k]}{r-1}$. We have the following equality:
    \[\bigsqcup_{F \in \Facetset{\uparrow}{\mbf z,\tau}} \nproj(F) = Z(\oline{C}_{\what{\tau}}(M)) + \nproj(M\mbf z) = \bigsqcup_{F \in \Facetset{\downarrow}{\mbf z,\tau}} \nproj(F).\]
\end{theorem}
\begin{proof}
Let $\mbf h$ be as defined in Proposition~\ref{prop:h}. By Proposition~\ref{prop:h}, the vector $\mbf h$ is in $\ker(\oline{C}_{\what{\tau}}(M))$. This means that $-\mbf h$ is also in $\ker(\oline{C}_{\what{\tau}}(M))$. The result then follows from applying Proposition~\ref{prop:doublecover} to both of the equalities in Corollary~\ref{cor:updownproj}. 
\end{proof}

\begin{corollary}\label{cor:DoubleIntersection}
Fix $\mbf z \in \Z^{r+k}$ and $\tau \in \binom{[r+k]}{r-1}$. Let $\mbf p$ be a point that is not on the boundary of any facet in $\Facetset{~}{\mbf z,\tau}$. Then, if $\mbf p$ is contained in some facet in $\Facetset{~}{\mbf z,\tau}$, it must be contained in exactly two of these facets. Furthermore, one of these facets must be in $\Facetset{\uparrow}{\mbf z,\tau}$ while the other must be in $\Facetset{\downarrow}{\mbf z,\tau}$.
\end{corollary}
\begin{proof}
    By Theorem~\ref{thm:sameproj}, the facets within $\Facetset{\uparrow}{\mbf z,\tau}$ or $\Facetset{\downarrow}{\mbf z,\tau}$ do not overlap and the union of their projections to the last $k$ coordinates cover the same space.
    Therefore, for a point \(\mbf x\) in the relative interior of some element of \(\Facetset{\uparrow}{\mbf z,\tau}\), the point \(\nproj (\mbf x) \) determines a unique element \(F\) of \(\Facetset{\uparrow}{\mbf z,\tau}\) which also contains the point \(\nproj (\mbf x) \).
    Moreover, by Proposition~\ref{prop:sameproj}, the projection of \(\mbf x\) to the first $r$ coordinates intersects \(\proj (F)\).
    Therefore the point \(\mbf x\) is contained in a unique element of \(\Facetset{\downarrow}{\mbf z,\tau}\).
    The reverse argumentation is analogous, and the result follows. 
    
\end{proof}

\begin{remark}
    We require that $\mbf p$ is not on the boundary of any facet for Corollary~\ref{cor:DoubleIntersection} because the projections of the facets to the first $r$ coordinates do not all have the same boundaries. We suspect that there may be more elegant ways to deal with these boundary points, but we were able to avoid them for our proofs in Section~\ref{sec:together}.  
\end{remark}

\begin{example} \label{ex:projtiles1}
Recall our running example with  \[M = \begin{bmatrix} 
3 & 2 & -4 & 1\\
1 & 0 & 2 & 2\\
2 & 0 & -1 & 1\\
0 & 1 & -2 & 3\\
\end{bmatrix},\] and $\mbf w = (1,1,1,1)^\top$. The set $\Facetset{~}{\mbf 0,\{2\}}$ is made up of the $6$ facets of the form $\widetilde{\mathcal F}( \mbf 0, \{2,j\}, j, s)$, where $j \in \{1,3,4\}$ and $s \in \{0,1\}$. 

Through direct calculation using Corollary~\ref{cor:updownalt}, one can show that 
\begin{align*} \Facetset{\uparrow}{\mbf 0,\{2\}} &= \left\{\widetilde{\mathcal F}( \mbf 0, \{1,2\}, 1, 0), \widetilde{\mathcal F}( \mbf 0, \{2,3\}, 3, 0), \widetilde{\mathcal F}( \mbf 0, \{2,4\}, 4, 1)\right\}, \text{ and}\\
\Facetset{\downarrow}{\mbf 0,\{2\}} &= \left\{\widetilde{\mathcal F}( \mbf 0, \{1,2\}, 1, 1), \widetilde{\mathcal F}( \mbf 0, \{2,3\}, 3, 1), \widetilde{\mathcal F}( \mbf 0, \{2,4\}, 4, 0)\right\}.
\end{align*}

See Figure~\ref{fig:projF} for the projections of the facets in each collection to their last $k$ coordinates. For both collections, the union of these projections forms the zonotope 
\[ Z(\oline C_{\{1,3,4\}}(M)) = Z(\begin{bmatrix} -2 &  1 & -1\\
0 &  2 & -3\end{bmatrix}).\] This matches what we would expect from Theorem~\ref{thm:sameproj}. 
\end{example}

Before stating the main result of the section, we quickly discuss the case when working with $\FacetsetO{~}{\mbf z,\gamma}$ for $\mbf z \in \Z^{r+k}$ and $\gamma \in \binom{[r+k]}{r+1}$. 

The main difference when working in this alternate perspective is that the role of the first $r$ and last $k$ coordinates swap. Furthermore, $\tau \cup j$ is replaced with $\gamma \setminus j$ and there are occasional sign changes. Nevertheless, the main ideas are analogous in this setting. In particular, $\FacetsetO{~}{\mbf z,\gamma}$ can also be divided into two subsets which one might call $\FacetsetO{\uparrow}{\mbf z,\gamma}$ and $\FacetsetO{\downarrow}{\mbf z,\gamma}$. 

\begin{theorem}\label{thm:1pointconstant}
    Fix $\mbf p \in \R^{r+k}$ and $\epsilon>0$ such that the line segment between $\mbf p$ and $\mbf p + \epsilon\w$ only intersects the set of all facets at a single point $\mbf q$, which is not on the boundary of any facet. Then,
    \[ f(\mbf p + \epsilon\w) = f(\mbf p).\]
\end{theorem}
\begin{proof}
    Recall that Lemma~\ref{lem:contrib} says that \[ f(\mbf p + \epsilon\w) - f(\mbf p) = \sum_{F \ni \mbf q} \wsgn(F)\tsgn(F),\] where the sum is taken over all facets of tiles in $\mbf T(M)$ which contain $\mbf q$. 

    By Corollary~\ref{cor:allFacets}, the set of all facets can be represented as the union of $\Facetset{~}{\mbf z,\tau}$ and $\FacetsetO{~}{\mbf z,\gamma}$ for all $\mbf z \in \Z^{r+k}$, $\tau \in \binom{[r+k]}{r-1}$, and $\gamma \in \binom{[r+k]}{r+1}$. 

    Combining Corollary~\ref{cor:DoubleIntersection} and Definition~\ref{def:FupFdown}, the total contribution to the sum is \(0\) for the facets contained in any particular $\Facetset{~}{\mbf z,\tau}$. 
    In particular, if any facets in $\Facetset{~}{\mbf z,\tau}$ contain $\mbf q$, there must be exactly one with $\wsgn(F)\tsgn(F) = 1$ and exactly one with $\wsgn(F)\tsgn(F) = -1$. 
    The same is true for the facets in $\FacetsetO{~}{\mbf z,\gamma}$, as discussed in the paragraph leading into this theorem statement. 

    Therefore, $f(\mbf p + \epsilon\w) - f(\mbf p)$ is equal to the sum of a collection of terms which are all zero. 
    This implies that $f(\mbf p + \epsilon\w) - f(\mbf p) =0$ and $f(\mbf p + \epsilon\w) = f(\mbf p) $.
\end{proof}

In the next section, we will show that Theorem~\ref{thm:1pointconstant} can be generalized to show that $f$ must be constant among all points in $\R^{r+k}$. 

\section{Putting it all Together}\label{sec:together}

\begin{lemma}\label{lem:offset}
    Consider any $\mbf p \in \R^{r+k}$. There exists some $\epsilon > 0$ such that $f(\mbf p + \epsilon\mbf w) = f(\mbf p)$ and $\mbf p + \epsilon\mbf w $ is not on the boundary of any tile. 
\end{lemma}
\begin{proof}
    Consider any $T \in \mbf T(M)$. By Definition~\ref{def:halfopen}, if $\mbf p \in T$, then $\mbf p + \epsilon\mbf w$ is on the interior of $T$ for all sufficiently small $\epsilon>0$. Furthermore, it also follows from this definition that if $\mbf p \not\in T$, then $\mbf p + \epsilon\mbf w \not\in T$ for all sufficiently small $\epsilon>0$. In particular, we can choose an $\epsilon>0$ such that $\mbf p \in T$ if and only if $\mbf p + \epsilon\mbf w \in T$. By the definition of $f$, this means that $f(\mbf p + \epsilon\mbf w) = f(\mbf p)$.
\end{proof}
\begin{lemma}\label{lem:constant}
The function \(f\) is constant on all of \(\R^{r+k}\).
\end{lemma}
\begin{proof}
    Let $\mbf p$ and $\mbf q$ be any two points on $\R^{r+k}$. 
    By Lemma~\ref{lem:offset}, it suffices to consider the case where $\mbf p$ and $\mbf q$ are not on the boundaries of any facets. 
    Consider a curve $C$ which connects $\mbf p$ to $\mbf q$ such that whenever $C$ crosses any facet, it crosses on the interior of that facet and parallel to $\w$. 
    By definition, \(f\) is constant when it does not cross a facet.
    By Theorem~\ref{thm:1pointconstant}, the value of $f$ is constant when \(C\) crosses a facet.
    Therefore the value of \(f\) is constant along \(C\), and the result follows. 
\end{proof}

We reiterate our main theorem, and provide its proof.

\vspace{ 2ex}
\noindent\textbf{Theorem~\ref{thm:mainthm}.} \textit{
 The function \(f(\mbf p) : \R^{r+k} \to \Z\), defined by
    \begin{equation*}f(\mbf p) := \left(\sum_{T \in \mbf T^+(M)}  \supp_{T}(\mbf p)\right) - \left(\sum_{T \in \mbf T^-(M)}  \supp_{T}(\mbf p)\right),\end{equation*}
    is constant with value $(-1)^{k}\sgn(\det(M))$. 
}

\begin{proof}
By Theorem~\ref{lem:constant}, we know that \(f\) is constant.
By Corollary~\ref{cor:avgvalue}, we know that the average value of \(f\) is $(-1)^{k}\sgn(\det(M))$.
Therefore \(f\) is constant with value $(-1)^{k}\sgn(\det(M))$.
\end{proof}

\section{Lower-Dimensional Slices}\label{sec:proj}
While it is possible to construct a signed tiling from any invertible matrix $M \in \R^{r+k}$, it is not immediately clear how to visualize such a tiling when $r+k>2$. One useful trick is to consider an $r$-dimensional slice of the tiling which fixes the last $k$-coordinates. 


\begin{definition}
    Let $S$ be a subset of $\R^{r+k}$. We will write $\rRestrict(S)$ for the intersection of $S$ with the plane whose last $k$ entries are $0$. 
\end{definition}

Notice that $\rRestrict(\R^{r+k})$ is naturally isomorphic to $\R^r$. Furthermore, since Theorem~\ref{thm:mainthm} is true for all $\mbf p \in \R^{r+k}$, the following is an immediate corollary. 

\begin{corollary}\label{cor:mainthmslice}
    The function \(f(\mbf p) : \R^r \to \Z\), defined by
    \begin{equation*}f(\mbf p) := \left(\sum_{T \in \mathbf T^+(M)}  \supp_{\rRestrict(T)}(\mbf p)\right) - \left(\sum_{T \in \mathbf T^-(M)}  \supp_{\rRestrict(T)}(\mbf p)\right),\end{equation*}
    is constant with value $(-1)^{k}\sgn(\det(M))$. 
\end{corollary}

We will see that if $M$ satisfies a few minor conditions, the periodic tiling structure is also preserved when restricting to $\rRestrict(\R^{r+k})$. Recall that for an $(r+k) \times (r+k)$ matrix $M$, we write $\oline{C}_{[r+k]}(M)$ for the matrix formed by the last $k$ rows of $-M$.
\begin{figure}
\centering
\begin{tabular}{ccc}
    
\includegraphics[width=.32\linewidth]{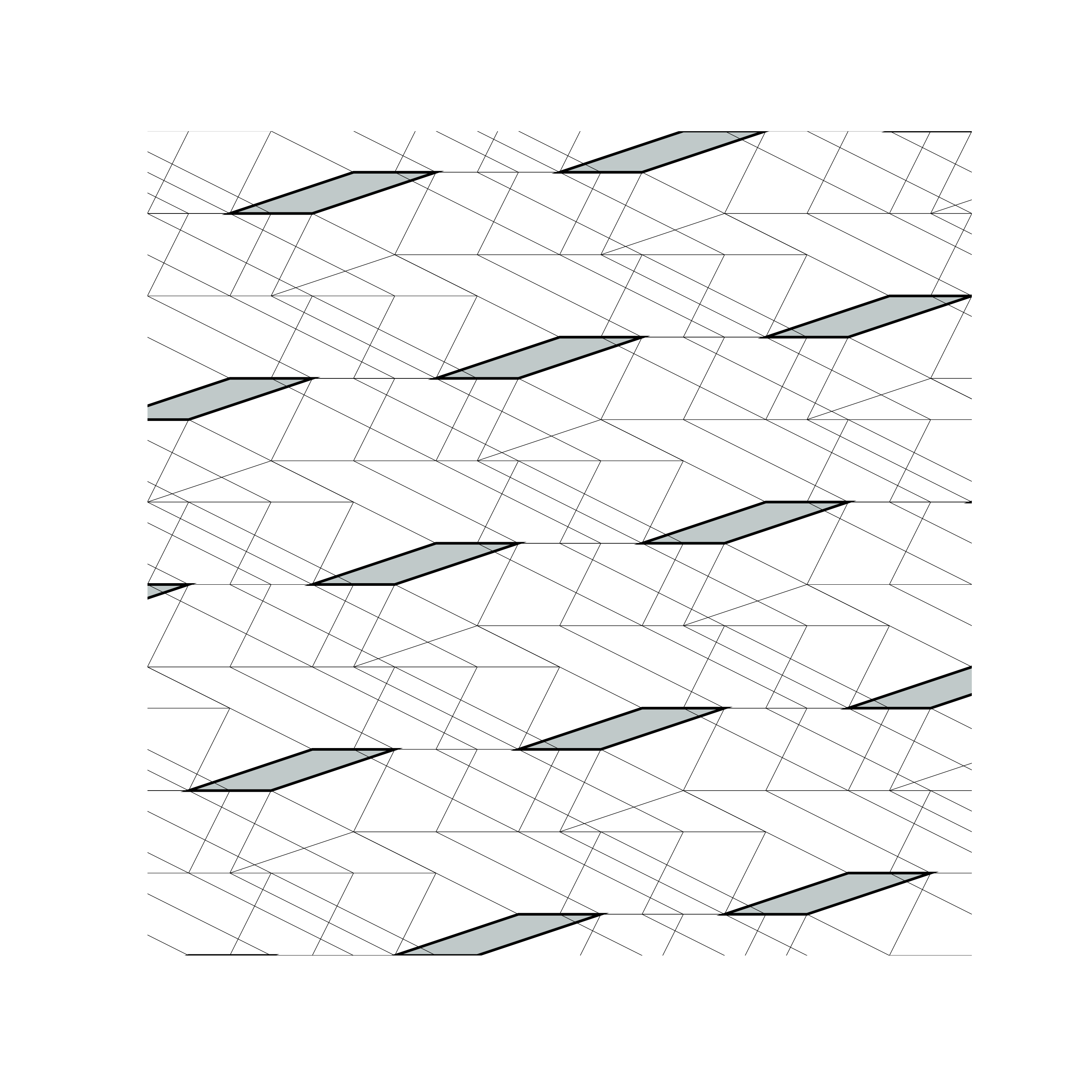} & 
\includegraphics[width=.32\linewidth]{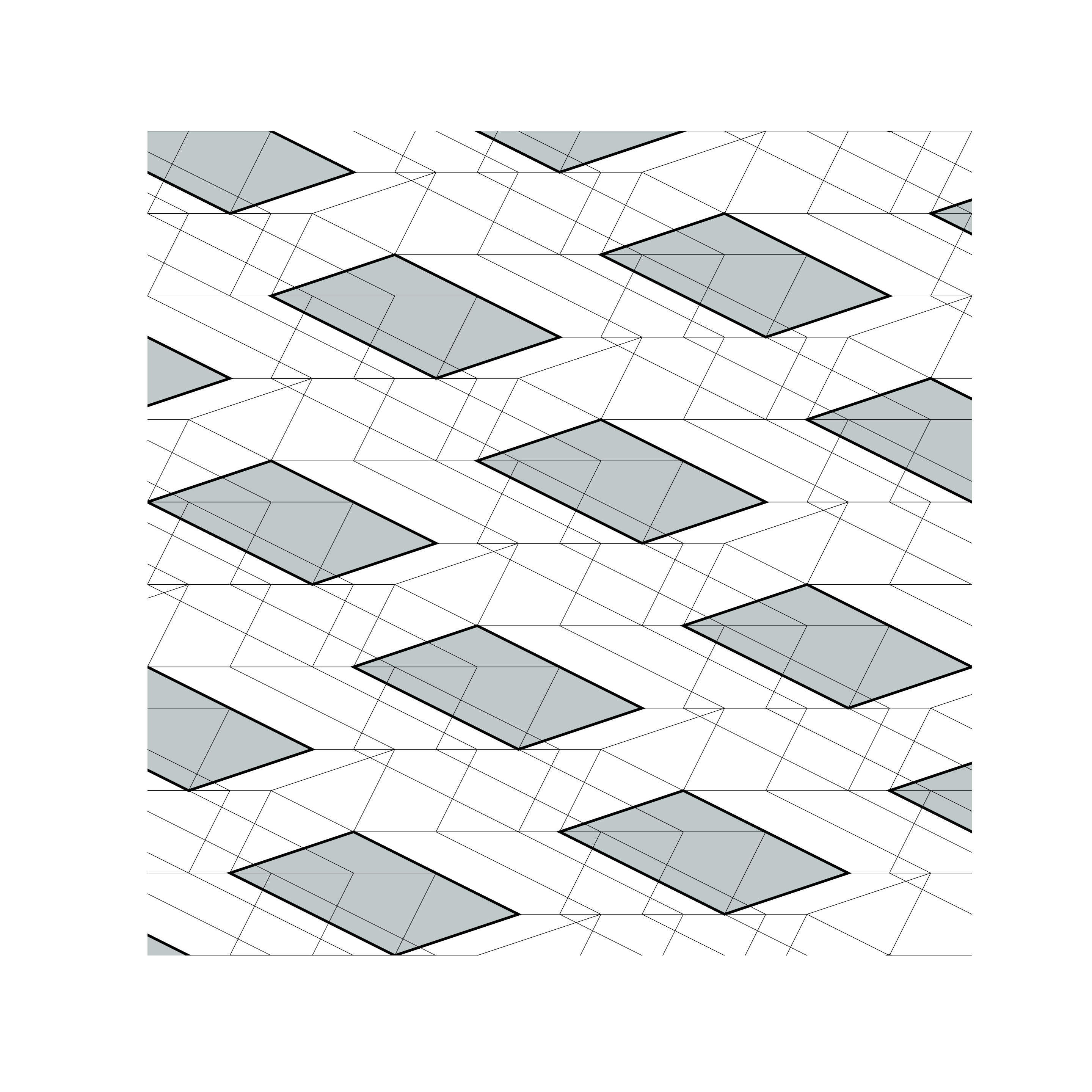}&
\includegraphics[width=.32\linewidth]{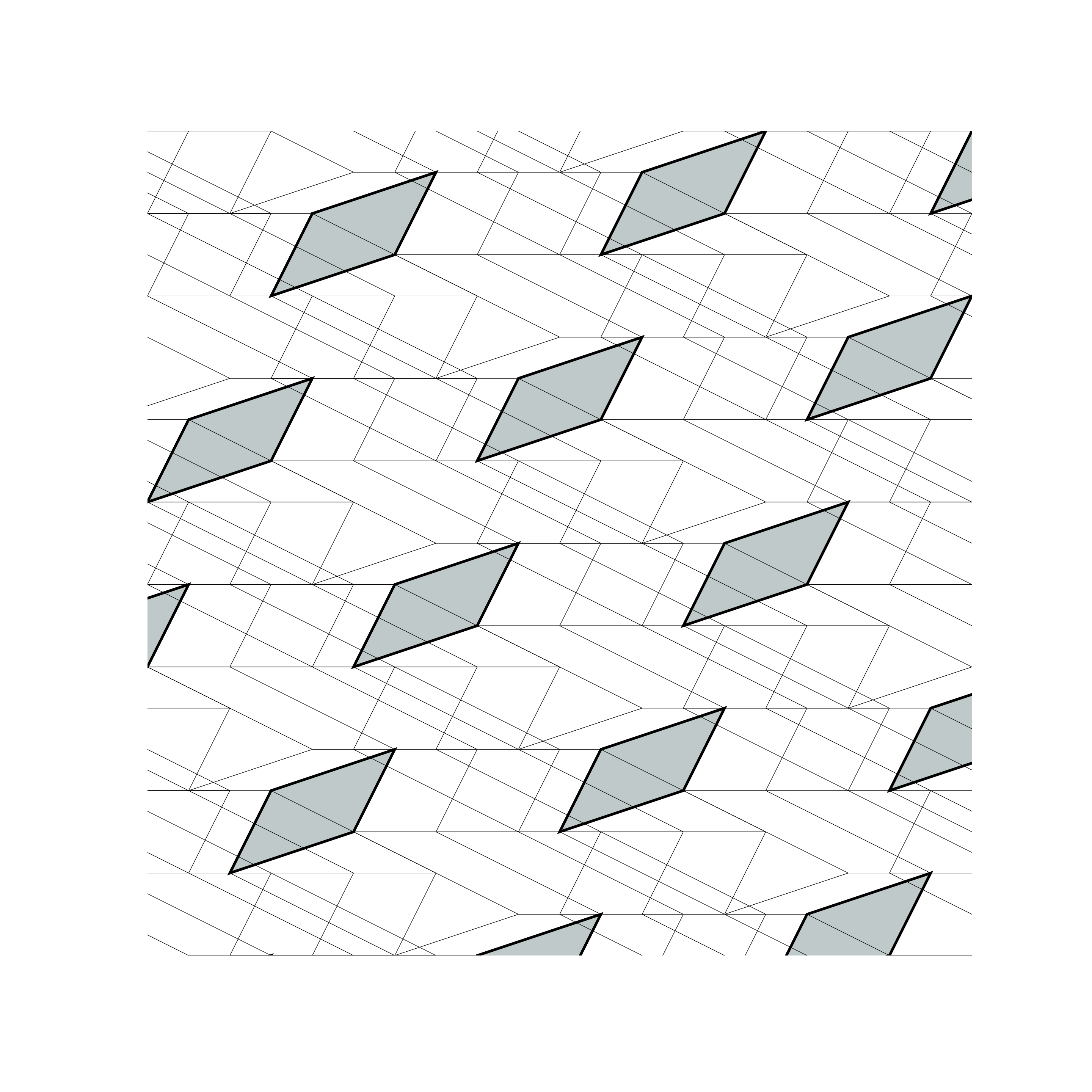} \\
$\det(C_{\{1,2\}}(M)) = -2$ & $\det(C_{\{1,3\}}(M)) = 10$ & $\det(C_{\{1,4\}}(M)) = 5$ \\
$\det(\oline {C}_{\{3,4\}}(M)) = -1$ & $\det(\oline {C}_{\{2,4\}}(M)) = -1$& $\det(\oline {C}_{\{2,3\}}(M)) = 1$\\
$\sgn(\{1,2,3,4\}) = 1$ & $\sgn(\{1,3,2,4\}) = -1$ & $\sgn(\{1,4,2,3\}) = 1$\\
   
\includegraphics[width=.32\linewidth]{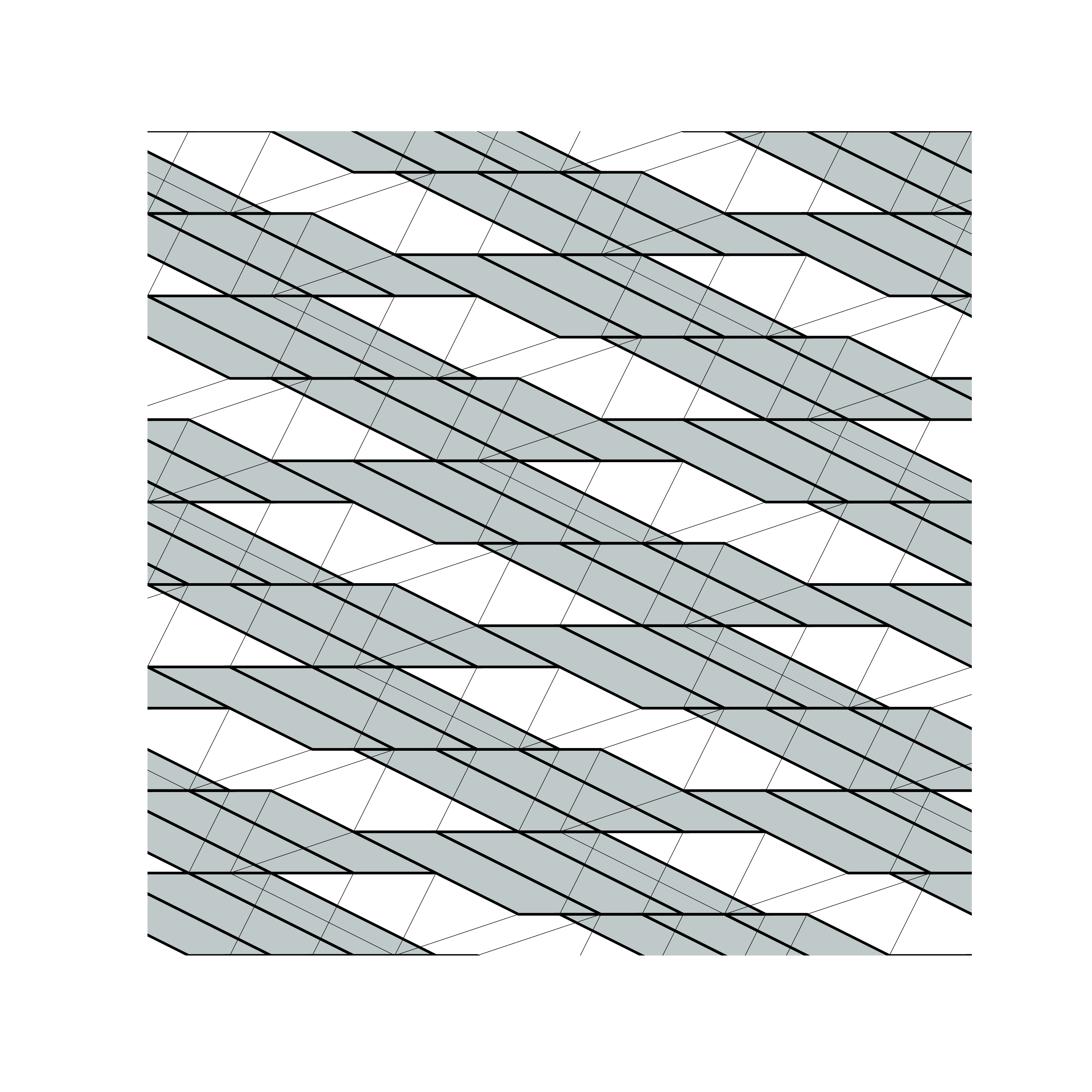}  & 
\includegraphics[width=.32\linewidth]{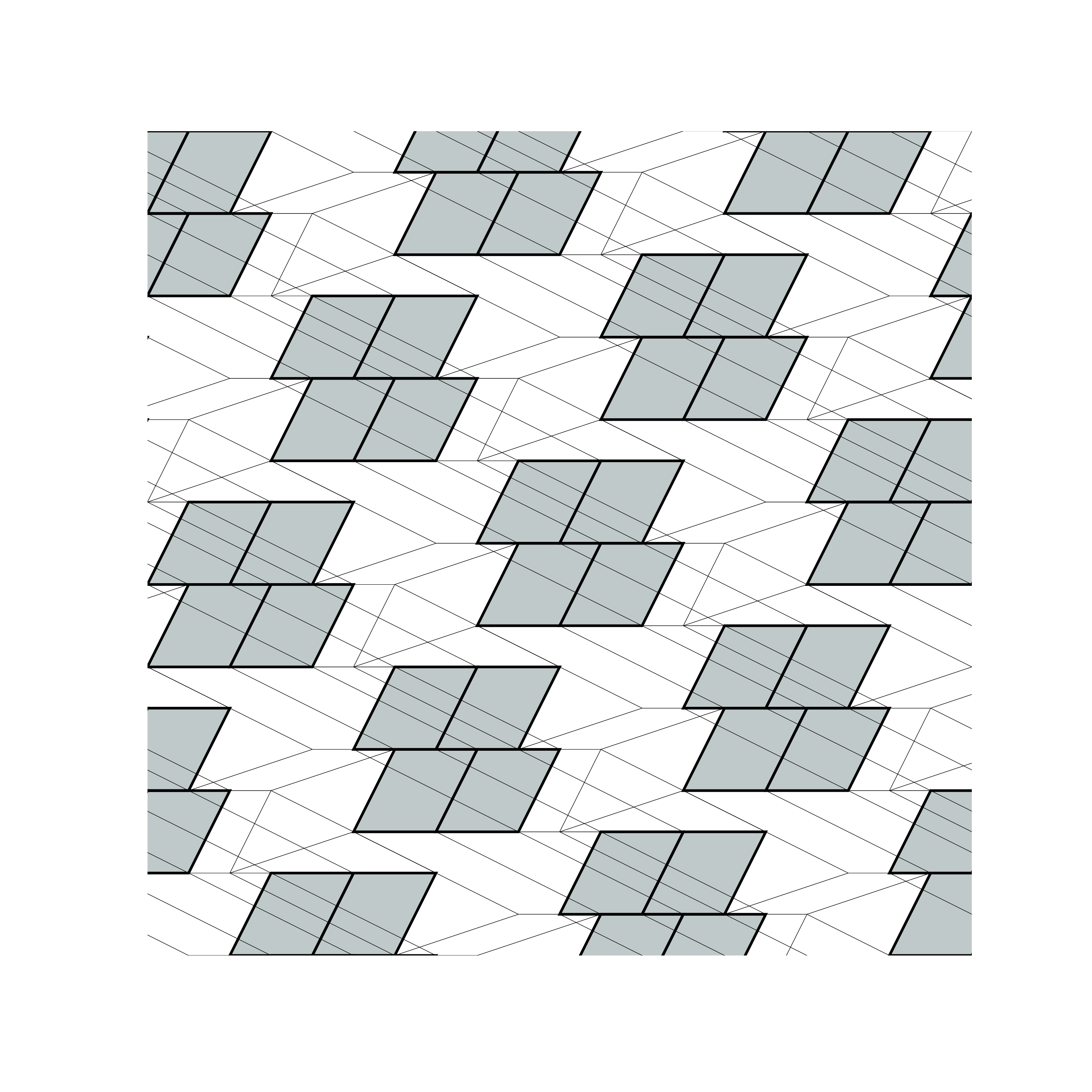} &
\includegraphics[width=.32\linewidth]{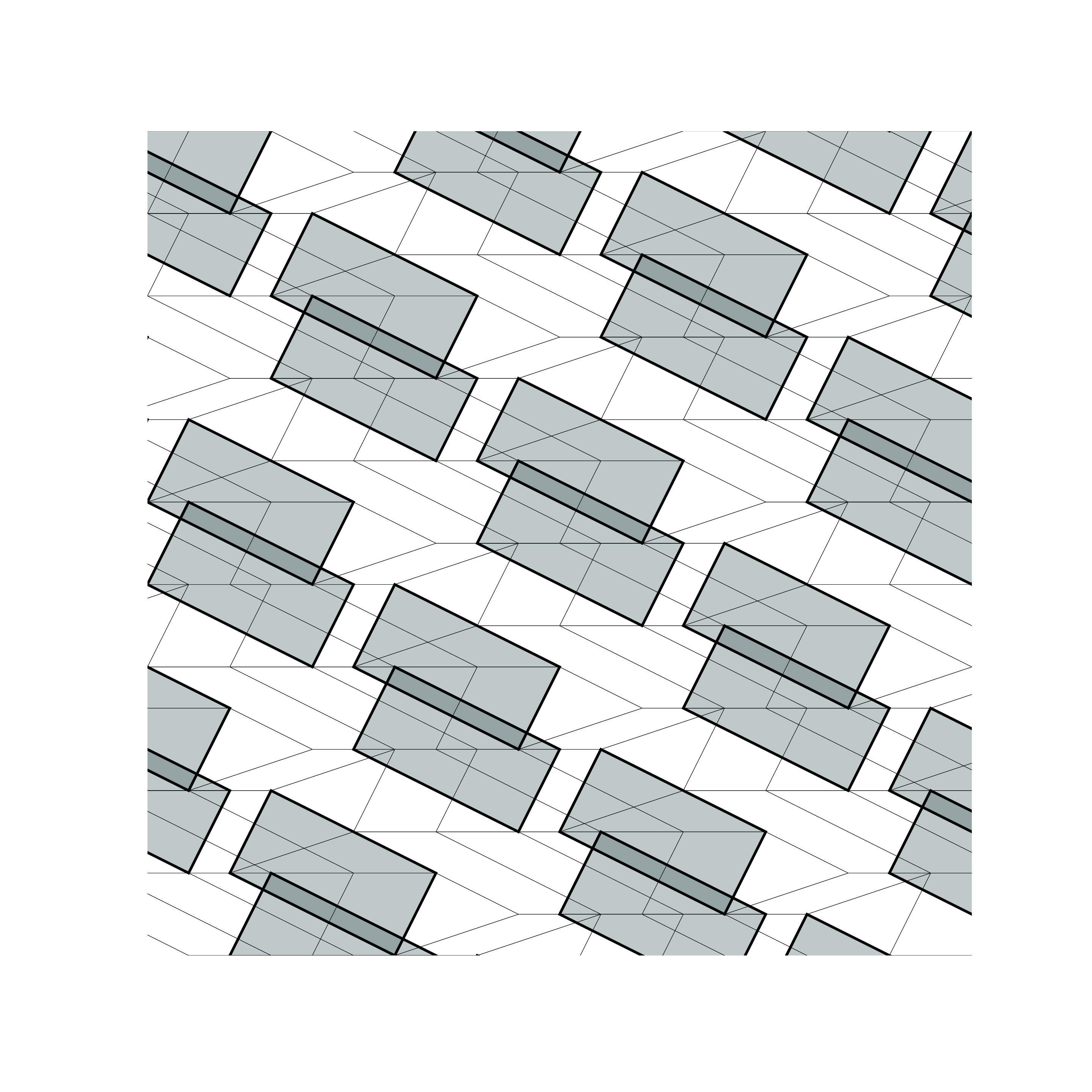}\\
$\det(C_{\{2,3\}}(M)) = 4$ & $\det(C_{\{2,4\}}(M)) = 4$ & $\det(C_{\{3,4\}}(M)) = -10$ \\
$\det(\oline {C}_{\{1,4\}}(M)) = 6$ & $\det(\oline {C}_{\{1,3\}}(M)) = -4$& $\det(\oline {C}_{\{1,2\}}(M)) = 2$\\
$\sgn(\{2,3,1,4\})=1$ & $\sgn(\{2,4,1,3\}) = -1$ & $\sgn(\{3,4,1,2\})=1$\\

\end{tabular}
\caption{Here we show the contributions of each of the six classes of tiles in Example~\ref{ex:projtiles1} to the tiling restricted to $\rRestrict(\R^{4})$. Notice that the magnitude of $\det(C_{\sigma}(M))$ corresponds to the size of each tile while the magnitude of $\det(\oline{C}_{\what\sigma}(M))$ corresponds to the relative frequency of each class of tiles.}
\label{fig:6projectedtiles}
\end{figure}
\begin{figure}
\includegraphics[width = .7\linewidth]{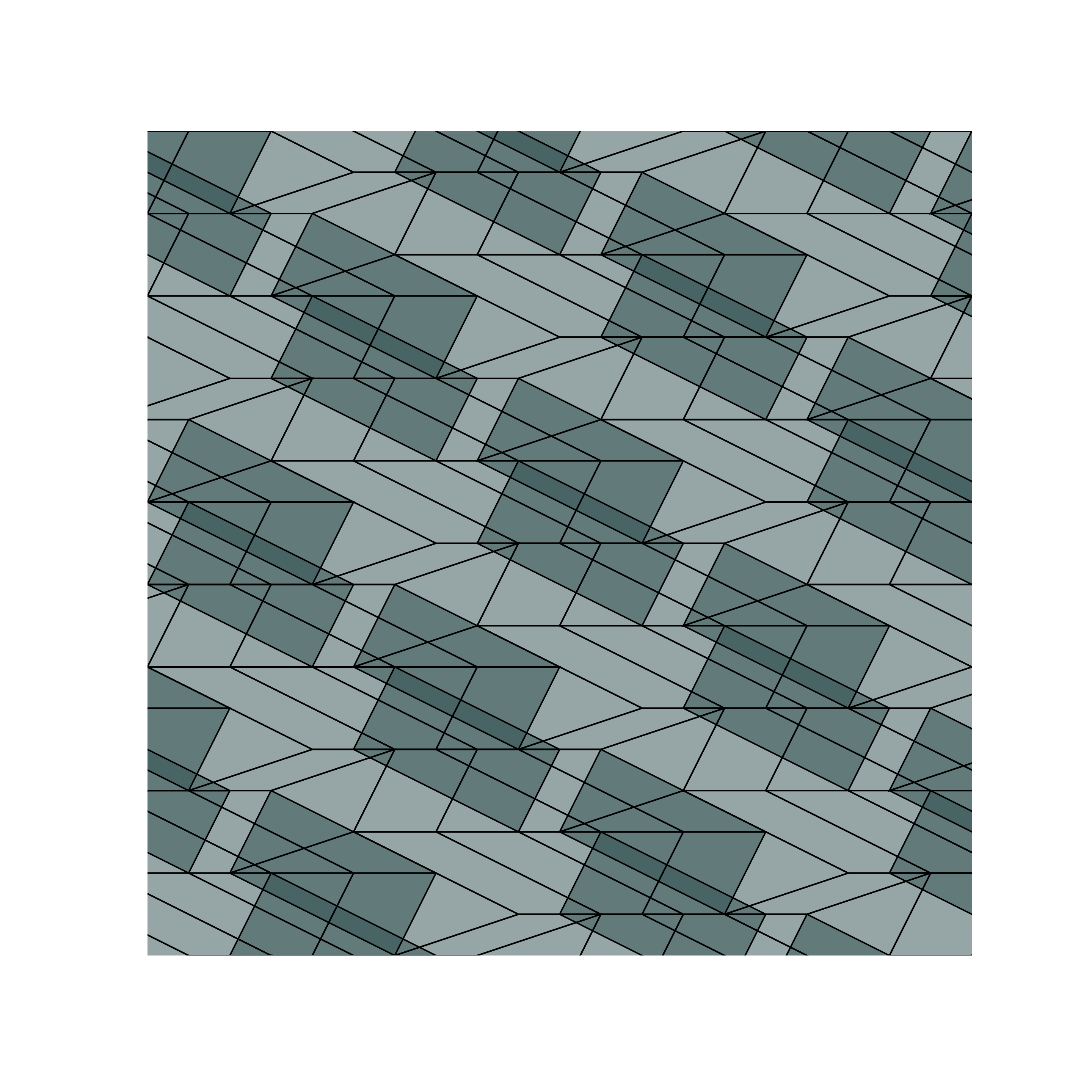}\\
\caption{This figure is formed by overlapping the first $5$ images in Figure~\ref{fig:6projectedtiles}. We showed in Example~\ref{ex:posnegtiles} that these are precisely the tiles in $\mbf T^+(M)$. If we ``subtract'' the final image in Figure~\ref{fig:6projectedtiles} from this picture, each point in $\R^2$ would be covered once.}\label{fig:phonecase}
\end{figure}
\begin{prop}
    Suppose that all of the entries in $\oline{C}_{[r+k]}(M)$ are integers. Further suppose that the GCD of the maximal minors of $\oline{C}_{[r+k]}(M)$ is $1$. Then, there exists some $r \times r$ matrix $B$ such that for each $\sigma \in \binom{[r+k]}{r}$, there is a collection $\mbf x_1,\dots,\mbf x_{\det(\oline{C}_{\what \sigma}(M))}$ of $\Z^r$ vectors which satisfy 
    \[\biguplus_{\mbf z \in \Z^{r+k}} \rRestrict(\mathcal T(\mbf z,\sigma)) = \biguplus_{i \in [\det(\oline{C}_{\what \sigma}(M))]}\left( \biguplus_{\mbf z \in \Z^r} \parp(C_\sigma(M)) + \mbf x_i + B\mbf z\right).\]

\end{prop}
\begin{proof} We only give a sketch of the proof here. For more discussion, see~\cite[Section 7]{multijections}. 

The main idea is that each tile of the form $\mathcal T(\mbf z,\sigma)$ maps to a collection of $|\det(\oline{C}_{\what \sigma}(M))|$ many translates of $\parp(C_\sigma(M))$ in $\rRestrict(\R^{r+k})$. This follows after applying column operations to $M$ in order to get a matrix of the form 
\[\begin{bmatrix} A & B \\ I_k & 0\end{bmatrix},\]
where $A$ and $B$ are integer matrices, $I_k$ is the $k\times k$ identity matrix, and $0$ is the $k \times r$ all zeros matrix. 

\end{proof}

\begin{remark}
Since we want to preserve the column lattice of $M$, our column operations do not allow for multiplication by anything other than $1$ or $-1$. This is the reason why we require the minors of $\oline{C}_{[r+k]}(M)$ to have GCD $1$; otherwise it would be impossible to get $I_k$ on the bottom left of the block matrix. Additionally, note that if we had fixed the last $k$ coordinates to be anything other than $\mbf 0$ in $\rRestrict(\R^{r+k})$, we would get the exact same tiling under some translation. 
\end{remark}

\begin{example} \label{ex:projtiles2}
Recall our running example with  \[M = \begin{bmatrix} 
3 & 2 & -4 & 1\\
1 & 0 & 2 & 2\\
2 & 0 & -1 & 1\\
0 & 1 & -2 & 3\\
\end{bmatrix},\] and $\w = (1,1,1,1)^\top$. Our signed tiling of $\R^4$ is made of translations of $6$ tiles. In particular, each tile is a translate of $\parp(S_{\sigma}(M))$ for $\sigma \in \binom{[4]}{2}$. After taking a slice in the first 2 dimensions, we still have $6$ kinds of tiles, but now they are translates of $\parp(C_\sigma(M))$ for $\sigma \in \binom{[4]}{2}$. See Figure~\ref{fig:6projectedtiles} for each class of tiles. 

We showed in Example~\ref{ex:posnegtiles} that the first 5 classes of tiles are made up of positive tiles while the final class is made up of negative tiles. 
Figure~\ref{fig:phonecase} gives an enlarged view of the collection of all positive tiles. 
By Corollary~\ref{cor:mainthmslice}, this is the same picture obtained by adding the negative tiles to the set of all points in $\R^2$. 
\end{example}

\section{Open Problems}\label{sec:questions}

The main motivation for this project was an attempt to gain a deeper understanding of a curious phenomenon (in particular Corollary~\ref{cor:traditionaltiling}). While we were successful at generalizing this statement to Theorem~\ref{thm:mainthm}, this new result is just as surprising. We expect that a deeper exploration of this problem will lead more surprises in the future, and we have several specific directions in mind the explore. 

Our initial approach when attempting to prove Theorem~\ref{thm:mainthm} was to consider an arbitrary point in $\R^{r+k}$ (or $\Pi(M)$) and compute which tiles contain this point. A direct proof of this form would give additional insight about the tiling, since it would allow us to calculate the number and type of tiles containing a given point. However, this method was more challenging than we expected, and we ended up relying on an indirect method by focusing on the facets and proving that $f$ is constant. 
\begin{open} What is the best algorithm to determine which tiles contain a given point? Can such an algorithm be used to give a more direct proof of Theorem~\ref{thm:mainthm}? \end{open}

Another promising method to prove Theorem~\ref{thm:mainthm} is to use Fourier analysis, applying similar methods to those used in~\cite{Ehrhart} (see also~\cite{Robins}). Perhaps these ideas could lead to a more elegant proof once the background is established. 

\begin{open} Is there a proof for Theorem~\ref{thm:mainthm} using Fourier analysis?\end{open} 

In addition to an alternate proof of the main theorem, we would also be interested in generalizing this result. As written, our construction relies on a choice of coordinates. While it should be possible to translate the statement into coordinate-free language, this is not a trivial task. Nevertheless, such a generalization would likely provide additional insight into the underlying phenomenon behind our construction. 

\begin{open} Is there a coordinate-free analogue to Theorem~\ref{thm:mainthm} or Corollary~\ref{cor:traditionaltiling}?\end{open}

\section*{Data Availability}
Data sharing is not applicable to this article as no data sets were generated or analysed during the current study.

\section*{Acknowledgments}  
We would first like to thank the organizers of the 33rd International Conference on Formal Power Series and Algebraic Combinatorics (FPSAC), which is where our exploration of this signed tiling began. We also would like to thank the anonymous reviewer for their excellent comments and suggestions. 

The first author would like to thank the organizers of the Combinatorial Coworkspace in Kleinwalsertal for providing an engaging environment to discuss this with many people, especially Marie-Charlotte Brandenburg.
They would also like to thank Bennet Goeckner for helpful comments on an earlier draft of this work.

The second author would like to thank Milen Ivanov and Sinai Robins for helping him to understand the analytic perspective as well as Julia Schedler for many lively discussions about the construction and visualizations. 
\bibliographystyle{plain}
\bibliography{biblio}

\begin{thebibliography}{1}

\bibitem{BBY}
Spencer Backman, Matthew Baker, and Chi~Ho Yuen.
\newblock Geometric bijections for regular matroids, zonotopes, and {Ehrhart}
  theory.
\newblock {\em Forum of Mathematics, Sigma}, 7:e45, 2019.

\bibitem{Ehrhart}
Ricardo Diaz and Sinai Robins.
\newblock The ehrhart polynomial of a lattice polytope.
\newblock {\em Annals of Mathematics}, 145(3):503--518, 1997.

\bibitem{multijections}
Alex McDonough.
\newblock A family of matrix-tree multijections.
\newblock {\em Algebraic Combinatorics}, 4(5):795--822, 2021.

\bibitem{alexthesis}
Alex McDonough.
\newblock {\em Higher-Dimensional Sandpile Groups and Matrix-Tree
  Multijections}.
\newblock PhD thesis, Brown University, 2021.

\bibitem{Nonlinear}
Alexei Morozov and Valery Dolotin.
\newblock {\em Introduction to non-linear algebra}.
\newblock World Scientific, 2007.

\bibitem{Robins}
Sinai Robins.
\newblock A friendly introduction to fourier analysis on polytopes.
\newblock {\em arXiv preprint arXiv:2104.06407}, 2021.

\end{thebibliography}
\end{document}